\documentclass[10pt,reqno]{amsart}
\usepackage[english]{babel}

\usepackage{enumitem}
\catcode`\"=\active \let"=\" \let\3=\ss
\usepackage{amsfonts}

\theoremstyle{plain}
  \newtheorem{theorem}{Theorem}[subsection]
  \newtheorem{corollary}[theorem]{Corollary}
  \newtheorem{lemma}[theorem]{Lemma}
  \newtheorem{proposition}[theorem]{Proposition}
  \newtheorem{proposition_within_section}{Proposition}[section]
\theoremstyle{definition}
  \newtheorem{definition}[theorem]{Definition}
  \newtheorem{examples}[theorem]{Examples}
  \newtheorem{example}[theorem]{Example}
  \newtheorem{rem}[theorem]{Remark}
  
  \newtheorem{open_problem}[theorem]{Open Problem}
  
  \newtheorem{example_within_section}[proposition_within_section]{Example}
  \newtheorem{rem_within_section}[proposition_within_section]{Remark}

\newcommand{\dist}{\mbox{{\rm dist}}}
\newcommand{\distPos}{\operatorname{d}_+}
\newcommand{\eins}{\mbox{\textbf{1}}}
\newcommand{\one}{\eins}

\newcommand{\asd}{\preceq_{\operatorname{a}}}
\newcommand{\das}{\succeq_{\operatorname{a}}}
\newcommand{\equ}{\simeq_{\operatorname{a}}}
\newcommand{\dx}{\:\mathrm{d}}
\DeclareMathOperator{\fix}{Fix}

\renewcommand{\phi}{\varphi}
\renewcommand{\epsilon}{\varepsilon}

\DeclareMathOperator{\id}{id}

\newcommand{\N}{\ensuremath{\mathbb{N}}}

\newcommand{\R}{\ensuremath{\mathbb{R}}}

\renewcommand{\L}{\mathcal{L}}
\renewcommand{\SS}{\mathcal{S}}
\newcommand{\T}{\mathcal{T}}

\setlist[enumerate]{leftmargin=*}
\setlist[itemize]{leftmargin=*}
\numberwithin{equation}{section}

\begin{document}
\title[Asymptotic Domination of Semigroups]{Long--Term Analysis of Positive Operator Semigroups via Asymptotic Domination}
\author{Jochen Gl\"uck}
\address{Institut f\"ur Angewandte Analysis, Universit\"at Ulm, 89069 Ulm, Germany}
\email{jochen.glueck@uni-ulm.de}
\author{Manfred P.~H.~Wolff}
\address{Mathematisches Institut, Universit\"at T\"ubingen, 72076 Tübingen, Germany}
\email{manfred.wolff@uni-tuebingen.de}
\dedicatory{Dedicated to Karl H.~Hofmann on the occasion of his 85.~birthday}
\date{\today}

\begin{abstract}
	We consider positive operator semigroups on ordered Banach spac\-es and study the relation of their long time behaviour to two different domination properties. First, we analyse under which conditions almost periodicity and mean ergodicity of a semigroup $\mathcal{T}$ are inherited by other semigroups which are asymptotically dominated by $\mathcal{T}$. Then, we consider semigroups whose orbits  asymptotically dominate a positive vector and show that this assumption is often sufficient to conclude strong convergence of the semigroup as time tends to infinity.
	
	Our theorems are applicable to time-discrete as well as time-continuous semigroups. They generalise several results from the literature to considerably larger classes of ordered Banach spaces.
\end{abstract}

\maketitle

\section{Introduction}

In this article we consider positive operator semigroups on ordered Banach spaces and study two topics which relate their long time behaviour to what we call \emph{asymptotic domination} (a notion which is made precise in Definition~\ref{def:asymptotic-domination}). Those two topics are:

\subsubsection*{(1) Almost periodicity and mean ergodicity}

Consider two positive operator semigroups $\T = (T_t)_{t \in J}$ and $\SS = (S_t)_{t \in J}$ on an ordered Banach space $X$ (where either $J = \N_0$ or $J = [0,\infty)$). Assume that the orbit of every positive vector $x$ under $\SS$ is dominated, up to a ``small error'', by its orbit under $\T$.

In \cite{Emelyanov2003} E.~Yu.~Emel'yanov and the second author studied unter which hypotheses mean ergodicity and almost periodicity is inherited from $\T$ to $\SS$. Many results in \cite{Emelyanov2003} were shown under the assumption that the positive $X_+$ is not only normal, but \emph{strongly normal} (see \cite[Definition~1]{Emelyanov2003}). Yet, K.~V.~Storozhuk demonstrated in \cite{Storozhuk2013} that there exist even finite-dimensional ordered Banach spaces whose cone is normal, but not strongly normal; this solved a problem posed in \cite{Emelyanov2003}. It is therefore natural to ask whether the results of \cite{Emelyanov2003} still hold for cones \emph{which are merely normal.} In Section~\ref{section:long-term-behaviour-of-asymp-dom-sg} of the present paper we show that this is the case for many of those results; this is achieved by replacing the strong normality of the cone with a general observation about asymptotic domination stated in Lemma~\ref{lem:key-lemma}.

\subsubsection*{(2) Lower bound theorems} In Section~\ref{section:lower-bound-theorems} we consider positive operator semigroups $\T$ whose orbits asymptotically dominate a time-independent positive vector $h$; we show that, in many situations, the existence of such a lower bound $h$ implies \emph{strong stability} of the semigroup, that means strong convergence of the semigroup as time tends to infinity. The history of such results goes back to a paper of Lasota and Yorke~\cite{Lasota1982} and since then many such \emph{lower bound theorems} were proved in the literature, some of them on AL-Banach lattices and some on non-commutative $L^1$-spaces (see Section~\ref{section:lower-bound-theorems} for detailed references). We generalise many of those results to a large class of ordered Banach spaces.

\subsubsection*{Organisation of the paper} In Section~\ref{section:prelim-and-basics} we introduce a bit of notation and recall several important concepts from operator theory that we need throughout the article. Moreover, we recall the definition of \emph{asymptotic domination} and prove a few elementary results for it. The content of Sections~\ref{section:long-term-behaviour-of-asymp-dom-sg} and~\ref{section:lower-bound-theorems} has already been described above. Motivated by the setting discussed in Section~\ref{section:lower-bound-theorems}, we close the paper with an Appendix containing some non-trivial examples of ordered Banach spaces whose norm is additive on the positive cone.

\section{Preliminaries and basic terminology} \label{section:prelim-and-basics}

\subsection{Preliminaries}

In this subsection we briefly recall some basics from the theory of ordered Banach spaces in order to fix our notation and for later reference (for the general theory of ordered Banach spaces see e.~g. \cite{Aliprantis2007,Batty1984,Schaefer66}).

\subsubsection*{Ordered Banach spaces}
By an \emph{ordered Banach space} we mean a pair $(X,X_+)$ where $X$ is a real Banach space and $X_+$ is a proper, closed, and convex cone; more precisely this means that $X_+$ is a closed subset of $X$ satisfying $X_+ \cap (-X_+) = \{0\}$, $X_+ + X_+ \subseteq X_+$, and $\alpha X_+ \subset X_+$ for all $\alpha \in [0,\infty)$. It is called the \emph{positive cone}, or simply the \emph{cone} in $X$. By abuse of notation we often write only $X$ to denote the ordered Banach space $(X,X_+)$. As is customary, we set $x \le y$ iff $y-x \in X_+$; the relation $\le$ is a partial order on $X$ which is compatible with addition and multiplication by nonnegative scalars due to the properties of $X_+$. For each $x \in X$ we have $x \in X_+$ iff $x \ge 0$, and in this case we call $x$ \emph{positive}. For $x,y \in X$ we write $x < y$ if $x \le y$ but $x \not= y$. We define the \emph{order interval} $[x,y] = \{z \in X: \, x \le z \le y\}$ for all $x,y \in X$. Whenever it is necessary to stress that an order interval $[x,y]$ is defined in the ordered Banach space $X$, we write more precisely $[x,y]_X$ instead of $[x,y]$. For all $x \in X$, we define
\begin{equation} \label{eq:distance-to-cone}
	\distPos(x) = \inf\{\|x-y\|: y \in X_+\}
\end{equation}
to be the distance of $x$ to the positive cone $X_+$. Note that, if $X$ is a Banach lattice or the self-adjoint part of the predual of a von Neumann algebra, then $\distPos(x) = \|x_-\|$ for each $x \in X$. 

On every ordered Banach space $X$ the function $\distPos: X \to [0,\infty)$ has a couple of useful properties which we list in the following proposition.

\begin{proposition} \label{prop:distance-to-positive-cone}
	Let $X$ be an ordered Banach space. The mapping $\distPos: X \to [0,\infty)$ has the following properties:
	\begin{enumerate}[label=(\alph*)]
		\item It is positively homogeneous, meaning that $\distPos(\alpha x) = \alpha \distPos(x)$ for all $x \in X$ and all $\alpha \in [0,\infty)$.
		\item It is subadditive, meaning that $\distPos(x+y) \le \distPos(x) + \distPos(y)$ for all $x,y \in X$.
		\item It is Lipschitz continuous with Lipschitz constant $1$, meaning that
		\begin{equation*}
			|\distPos(x) - \distPos(y)| \le \|x-y\|
		\end{equation*}
		for all $x,y \in X$.
	\end{enumerate}
\end{proposition}
\begin{proof}
	Property~(a) follows from the fact that $\alpha X_+ \subseteq X_+$ for all $\alpha \in [0,\infty)$, and property~(b) follows from the fact that $X_+ + X_+ \subseteq X_+$. Assertion~(c) is in fact true for the distance to an arbitrary set in $X$, and not only for the distance to $X_+$; as the proof of this is elementary, we omit the details.
\end{proof}
The cone $X_+$ is called \emph{generating} if $X_+-X_+ = X$ holds. Having a generating cone is a very weak condition which is met by most of the ordered Banach spaces that one usually encounters. The cone is
 called \emph{normal} if the order interval $[x,y]$ is norm bounded for all $x,y \in X$.  The class of ordered Banach spaces having a normal cone is still sufficiently large to contain a wealth of important examples. For instance, each of the following ordered Banach spaces possesses have a generating and normal cone: all Banach lattices; the self-adjoint part of every $C^*$-algebra; the dual and the pre-dual of the self-adjoint part of every von Neumann algebra. In particular, non-commutative $L^1$-spaces with respect to a semi-finite trace belong to the list of examples since they are preduals of semi-finite von Neumann algebras, see \cite[p.~320]{Takesaki1979}.

If the positive cone of an ordered Banach space $X$ is generating, than there exists a constant $C > 0$ such that every $x \in X$ can be decomposed as
\begin{equation} \label{eq:bounded-decomposition}
	x = y-z \quad \text{for two vectors } y,z \in X_+ \text{ fulfilling } \|y\|, \|z\| \le C\|x\|.
\end{equation}
For a proof of this decomposition property (which traces back to V. Klee) we refer, for instance, to \cite[Proposition~1.1.2]{Batty1984} or \cite[Theorem~2.37]{Aliprantis2007}.

The positive cone in the ordered Banach space $X$ is normal if and only if there exists an equivalent norm on $X$ which is \emph{monotone}, meaning that $\|x\| \le \|y\|$ for all $x,y \in X$ fulfilling $0 \le x \le y$;

\subsubsection*{Duality and positive operators}

We denote the space of all bounded linear operators between two Banach spaces $X$ and $Y$ by $\L(X;Y)$ and we abbreviate $\L(X) := \L(X;X)$. The dual space of a Banach space $X$ is denoted by $X'$. For all $x \in X$ and $x' \in X'$ we use the notation $\langle x', x\rangle := x'(x)$.

Let $X,Y$ be two ordered Banach spaces. A linear operator $T: X \to Y$ is called \emph{positive}, $T \ge 0$ for short, if $T(X_+) \subseteq Y_+$. In case that the positive cone in $X$ is generating, every positive linear operator $T: X \to Y$ is automatically bounded; see e.g.\ \cite[Theorem~2.8]{Arendt2009} or \cite[Corollary~2.33]{Aliprantis2007}. The set of all positive operators $T \in \L(X;Y)$ is denoted by $\mathcal{L}(X;Y)_+$ and, again, we abbreviate $\L(X)_+ := \L(X;X)_+$. If $x \in X$ and $T \in \L(X)_+$, then we have
\begin{align} \label{eq:positive-operator-and-distance-to-cone}
	\distPos(Tx) \le \|T\| \distPos(x).
\end{align}

Throughout, we endow $\R$ with the positive cone $\R_+ := [0,\infty)$ which renders it an ordered Banach space. By the terminology introduced above, a linear mapping $\varphi: X \to \R$ is positive iff $\varphi(x) \ge 0$ for all $x \in X_+$. In case that $X_+$ is generating in $X$, every such $\varphi$ is automatically continuous by the previous paragraph. We define $X'_+ := \{x' \in X': \; x' \text{ is positive}\}$ and, for every $x' \in X'$, we write $x' \ge 0$ iff $x' \in X'_+$; we call such elements $x' \in X'_+$ \emph{positive functionals} on $X$. Obviously, $X'_+ = (-X)^\circ$ (the polar of $-X$ in $X'$, see e.g.\ \cite[Chapter~V]{Schaefer66}), hence $X'_+$ is a closed convex cone. If $X_+$ is generating in $X$, or more generally if $X_+ - X_+$ is dense in $X$, then $X'_+$ is proper and normal so $X'$ becomes an ordered Banach space with respect to the cone $X'_+$. In the following proposition we collect two useful observations on duality and positivity; we denote the \emph{dual operator} of a bounded linear operator $T$ between two Banach spaces by $T'$.

\begin{proposition} \label{prop:duality-and-positivity}
	Let $X,Y$ be ordered Banach spaces.
	\begin{enumerate}[label=(\alph*)]
		\item A vector $x \in X$ is positive if and only $\langle x', x \rangle \ge 0$ for all $x' \in X'_+$.
		\item Let $T \in \L(X;Y)$. We have $T \ge 0$ if and only if $T' \ge 0$.
	\end{enumerate}
\end{proposition}

In assertion~(b) of this proposition we use the notation $T' \ge 0$ to indicate that $T'X'_+ \subseteq Y'_+$, and we use the notation even if the cones in $X$ and $Y$ are not generating (in which case $X'$ and $Y'$ are not ordered Banach spaces in the sense defined above).

\begin{proof}[Proof of Proposition~\ref{prop:duality-and-positivity}]
	(a) It follows from the bipolar theorem \cite[Theorem \\IV.1.5]{Schaefer66} that  $X_+ = (-X'_+)^\circ$ (where ${}^\circ$ denotes the polar with respect to the duality $(X,X'$)). The implication ``$\Rightarrow$'' in~(b) is obvious and the implication ``$\Leftarrow$'' follows from~(a).
\end{proof}

\begin{corollary} \label{cor:positive-non-annihilating-functional}
	Let $X$ be an ordered Banach space.
	\begin{enumerate}
		\item[(a)] For each $x \in X \setminus \{0\}$ there exists a positive functional $x' \in X'$ such that $\langle x', x\rangle \not= 0$.
		\item[(b)] For each $0 < x \in X$ there exists a positive functional $x' \in X'$ such that $\langle x', x\rangle > 0$.
	\end{enumerate}
\end{corollary}
\begin{proof}
	(a) Assume that $\langle x', x \rangle = 0$ for each positive $x' \in X'$. By Proposition~\ref{prop:duality-and-positivity}\\(a) this implies that $x \ge 0$ and $-x \ge 0$, so $x = 0$.
	
	(b) This is an immediate consequence of~(a).
\end{proof}

Let $X$ be an ordered Banach space with generating cone. Recall that the dual cone $X'_+$ is generating in $X'$ if and only if the cone $X_+$ in $X$ is normal (see \cite[Corollary~3 on p.\,220]{Schaefer66} or \cite[Theorem~2.40]{Aliprantis2007}). In this case the bidual space $X''$ becomes an ordered Banach space, too, namely with respect to the cone $X''_+ := \{x'' \in X'': \; \langle x'', x' \rangle \ge 0 \text{ for all } x' \in X'_+\}$; the cone $X'_+$ in $X'$ is automatically normal \cite[Theorem~2.42]{Aliprantis2007}, so the bi-dual cone $X''_+$ is even generating in $X''$. The evaluation map $j: X \hookrightarrow X''$ is \emph{bipositive}, meaning that $j(x) \ge 0$ if and only if $x \ge 0$; this follows from Proposition~\ref{prop:duality-and-positivity}(a). We may thus consider $X$ as a subspace of $X''$ by means of evaluation without violating the order structure on those spaces.

\subsubsection*{Projection bands in the bidual}

In several occasions in Section~\ref{section:lower-bound-theorems} we will consider ordered Banach spaces who behave, in a sense, similarly to so-called \emph{KB-Banach lattices}. To introduce such a concept for general order Banach spaces, we need a bit of preparation. Let $X$ be an ordered Banach space with generating cone. We call a projection $P \in \L(X)$ a \emph{band projection} if both $P$ and its complementary projection $\id_{X} - P$ are positive. A projection $P \in \L(X)$ is a band projection if and only if $0 \le Px \le x$ for all $x \in X_+$.

\begin{proposition}
	Let $X$ be an ordered Banach space with generating cone and let $P,Q \in \L(X)$ be two band projections with the same range. Then $P = Q$.
\end{proposition}
\begin{proof}
	We first show that $Qx = 0$ for every $0 \le x \in \ker P$. Indeed, for such a vector $x$ we have $0 \le Qx \le x$ and thus, as $P$ acts as the identity on the range of $Q$,
	\begin{align*}
		0 \le Qx = PQx \le Px = 0.
	\end{align*}
	Hence, $Qx = 0$. Now we argue that this implies $\ker P \subseteq \ker Q$. In fact, we only have to note that every vector in $\ker P$ is a linear combination of two positive vectors in $\ker P$ since $X_+$ is generating in $X$ and since $\ker P$ is the range of the positive projection $\id_X - P$.
	
	Since our assumptions on $P$ and $Q$ are symmetric, we conclude that also $\ker Q \subseteq \ker P$, so $P$ and $Q$ do not only have the same range, but also the same kernel. This proves $P = Q$, as claimed.
\end{proof}

As a consequence of the preceding proposition, a vector subspace $V$ of $X$ can only be the range of at most one band projection $P \in \L(X)$; in this case, we call $P$ \emph{the} band projection onto $V$. Now we can define the following class of ordered Banach spaces that we mentioned above.

\begin{definition} \label{def:band-in-bidual}
	Let $X$ be an ordered Banach space with generating and normal cone and consider $X$ as a subspace of its bidual $X''$ by means of evaluation. We say that $X$ is \emph{a projection band in its bidual} if there exists a band projection $P \in \L(X'')$ with range $X$.
\end{definition}

A Banach lattice $X$ is a projection band in its bi-dual if and only if every norm-bounded increasing sequence (equivalently: net) in $X$ converges in norm; such spaces are called \emph{KB-spaces}; \cite[Definition~2.4.11 and Theorems~2.4.12 and~1.2.9]{Meyer-Nieberg1991} and \cite[Example~7 on p.\,92 and Proposition~II.5.15]{Schaefer74}. For instance, every AL-Banach lattice is a KB-space. On ordered Banach spaces, though, the situation is more complicated -- compare Remark~\ref{rem:kb-for-non-lattices}. Important examples of ordered Banach spaces which are projection bands in their biduals are preduals of
von Neumann-algebras (see \cite[pp.~126--127]{Takesaki1979} or \cite[Proposition 1.17.7]{Sakai1971}).

Ordered Banach spaces which are projection bands in their biduals have quite nice regularity properties as the following proposition shows:

\begin{proposition} \label{prop:band-in-bidual-implies-weakly-compact-intervals}
	Let $X$ be an ordered Banach space with generating and normal cone. If $X$ is a projection band in its bidual, then every order interval in $X$ is weakly compact.
\end{proposition}
\begin{proof}
	Let $P \in \L(X'')$ be the band projection onto $X$ and let $Q := \id_{X''}-P$ denote the complementary projection. Let $x,y \in X$. We first note that $[x,y]_X = [x,y]_{X''}$. Indeed, the inclusion ``$\subseteq$'' is obvious and, conversely, every $z'' \in [x,y]_{X''}$ fulfils $0 = Qx \le Qz'' \le Qy = 0$, so $z'' \in X$.
	
	Since $[x,y]_X = [x,y]_{X''}$ is compact in $X''$ with respect to the weak${}^*$-topology, it is compact with respect to the weak topology in $X$.
\end{proof}

\subsubsection*{Operator semigroups}

Let $X$ be a Banach space and let either $J=\N_0:=\{0,1,2,..\}$ or $J = [0,\infty)$. We call a family $\T = (T_t)_{t \in J}$ of bounded linear operators on $X$ a \emph{semigroup} -- or, more precisely, a one-parameter operator semigroup -- on $X$ if $T_0 = \id_X$ and $T_{s+t} = T_sT_t$ for all $s,t \in J$. The semigroup $\T$ is called \emph{bounded} if $\sup_{t \in J} \|T_t\| < \infty$. We call $\T$ \emph{(weakly) almost periodic} in case that the set $\{T_t: \, t \in J\}$ is relatively compact in $\L(E)$ with respect to the strong (weak) operator topology. Every almost periodic semigroup is weakly almost periodic and every weakly almost periodic semigroup is bounded by the uniform boundedness principle. A vector $x \in X$ is called a \emph{fixed point} or a \emph{fixed vector} of $\T$ if $T_tx = x$ for all $t \in J$; the subspace of $X$ consisting of all fixed points of $\T$ is called the \emph{fixed space} $\fix(\T)$ of $\T$. In case that $X$ is an ordered Banach space, we call the semigroup $\T$ \emph{positive} if $T_t$ is a positive operator for every $t \in T$. If $J = [0,\infty)$ and the mapping $J \ni t \mapsto T_t \in \L(X)$ is strongly continuous, then $\T$ is called a $C_0$-semigroup.

Now, let $\T = (T_t)_{t \in J}$ be a semigroup on a Banach space $X$ and assume that $\T$ is a $C_0$-semigroup in case that $J = [0,\infty)$. For each $t \in J \setminus \{0\}$ we denote by $A_t(\T)$ the $t$-th Ces{\`a}ro mean of $\T$, i.e.\ $A_t(\T) = \frac{1}{t} \sum_{k=0}^{t-1}T_k$ for each $t \in \N := \{1,2,3,\dots\}$ in case that $J = \N_0$, and $A_t(\T) = \frac{1}{t} \int_0^t T_s \, \dx s$ for each $t \in (0,\infty)$ in case that $J = [0,\infty)$ (where the integral is to be understood in the strong sense). The semigroup $\T$ is called \emph{mean ergodic} if $A_t(\T)$ converges strongly as $t \to \infty$. If $\T$ is weakly almost periodic, then it is mean ergodic. An operator $T \in \L(X)$ is called \emph{mean ergodic} if the semigroup $(T^n)_{n \in \N_0}$ is mean ergodic. The semigroup $\T$ is called \emph{Ces{\`a}ro bounded} if the net $(A_t(\T))_{t \in J \setminus \{0\}}$ of its Ces{\`a}ro means is bounded with respect to the operator norm, and an operator $T$ is called \emph{Ces{\`a}ro bounded} if the semigroup $(T^n)_{n \in \N_0}$ is so. For a detailed treatment of all these notions we refer for instance to the monographs \cite{Krengel1985, Engel2000,Emelyanov2007,Eisner2015}.

\subsection{Basic terminology: asymptotic domination}

As pointed out in the introduction, the present paper deals with the interplay between the long term behaviour of positive semigroups and the concept of \emph{asymptotic domination}. In this section we collect several basic results about this property, the most important one being certainly Lemma~\ref{lem:key-lemma}. Let us begin by recalling the definition of asymptotic domination, both for vector-valued and for operator-valued functions.

\begin{definition} \label{def:asymptotic-domination}
	Let $J$ be a subset of $\R$ which is not bounded above and let $X$ be an ordered Banach space.
	\begin{enumerate}[label=(\alph*)]
		\item Let $f, g: J \mapsto X_+$ be two mappings. We say that $g$ {\em dominates $f$ asymptotically} or that $f$ is {\em asymptotically dominated by $g$} if $\lim_{t \to \infty} \distPos(g(t)-f(t)) = 0$. We denote this property by $g \das f$ or by $f \asd g$.
		
		\item Let $\SS, \T: J \mapsto \mathcal{L}(X)_+$ be two mappings. We say that $\T$ {\em dominates $\SS$ asymptotically} or that $\SS$ is {\em asymptotically dominated by $\T$} if the mapping $t \mapsto \T(t)x$ dominates the mapping $t \mapsto \SS(t)x$ asymptotically for all $x \in X_+$. We denote this property by $\SS \das \T$ or by $\T \asd \SS$.
		
		\item Two mappings $f,g:J \mapsto X_+$ are called {\em asymptotically equivalent}, which we denote by $f \equ g$, if $f \asd g$ and $g \asd f$.
		
		Likewise, we call two mappings $\SS,\T: J \to \mathcal{L}(X)_+$ {\em asymptotically equivalent}, also denoted by $\SS \equ \T$, if $\SS \asd T$ and $\T \asd \SS$.
	\end{enumerate}
\end{definition}

We are mainly interested in the case where $J$ is one of the two sets $[0,\infty)$ and $\N_0$. When we consider Ces{\`a}ro means of semigroups, the cases $J = \N$ and $J = (0,\infty)$ will also occur. Yet, the particular choice of $J$ is not important for any of the basic properties of asymptotic domination that we consider in this section.

It is sometimes convenient to use the following convention: if $f,g: J \to X_+$ are mappings and $x \in X_+$ is a fixed vector, then we shall write $x \das f$ or $f \asd x$ to say that $f$ is asymptotically dominated by the constant function with value $x$; the notations $g \das x$ and $g \asd x$ are defined similarly. Likewise, we use the notations $P \das \T$ or $\T \asd P$ and $\SS \das P$ or $P \asd \SS$ for functions $\T,\SS: J \to \mathcal{L}(X)_+$ and for a fixed operator $P \in \mathcal{L}(X)_+$.

The following elementary characterisation of asymptotic domination is quite useful and we will often employ it tacitly. It is taken from \cite[Proposition~2.1.10]{Emelyanov2007}; for the convenience of the reader we include the proof.

\begin{proposition} \label{prop:asymptotic-domination-via-error-function}
	Let $J$ be a subset of $\R$ which is not bounded above, let $X$ be an ordered Banach space with generating cone and consider to mappings $f,g: J \to X_+$. The following assertions are equivalent:
	\begin{enumerate}
		\item $g$ asymptotically dominates $f$.
		\item There exists a function $r: J \to X_+$ satisfying $\lim_{t \to \infty} r(t) = 0$ and $f(t) \le g(t) + r(t)$ for all $t \in J$.
	\end{enumerate}
\end{proposition}
\begin{proof}
	If~(ii) is fulfilled, then
	\begin{align*}
		\distPos\big(g(t)-f(t)\big)) \le \big\|[g(t) - f(t)] - [g(t) - f(t) + r(t)]\big\| = \|r(t)\| \to 0
	\end{align*}
	as $t \to \infty$, so~(i) holds; note that this implication does not depend on the assumption that $X_+$ be generating.
	
	Now, assume on the other hand that (i) is fulfilled. For every $t \in J$ there exists a vector $u(t) \in X_+$ satisfying $\big\|[g(t) - f(t)] - u(t)\big\| \le   \distPos\big(g(t)-f(t)\big)+ 2^{-t}$. Let the constant $C > 0$ be as in~\eqref{eq:bounded-decomposition}. For every $t \in J$ we can hence decompose the vector $[g(t) - f(t)] - u(t)$ as
	\begin{align*}
		[g(t) - f(t)] - u(t) = y(t) - z(t)
	\end{align*}
	where $y(t),z(t) \in X_+$ fulfil $\|y(t)\|, \|z(t)\| \le C \cdot \big(\distPos(g(t) - f(t))+2^{-t}\big)$. The function $r := z$ thus fulfils the properties in~(ii).
\end{proof}

We are mainly interested in asymptotic domination between two mappings $\SS,\T: J \to \mathcal{L}(X)_+$ in the special case where $\SS$ and $\T$ are positive operator semigroups. Before we give some examples where this phenomenon occurs, we prove a few simple facts for the relations $\asd$ and $\equ$.

\begin{proposition} \label{prop:asd-elementary-properties}
	Let $J$ be a subset of $\R$ which is not bounded above, let $X$ be an ordered Banach space and consider two mappings $f,g: J \to X_+$.
	\begin{enumerate}[label=(\alph*)]
		\item The relation $\asd$ is transitive and the relation $\equ$ is an equivalence relation.
		\item Let $(t_k)_{k \in \N_0}$ be a sequence in $J$ converging to $\infty$. If $g$ asymptotically dominates $f$, then the mapping $k \mapsto g(t_k)$ asymptotically dominates the mapping $\N_0 \ni k \mapsto f(t_k) \in X_+$.
		\item If $f(t)$ and $g(t)$ both converge to a vector $x \in X_+$ as $t \to \infty$, then $f \equ g \equ x$.
	\end{enumerate}
\end{proposition}
\begin{proof}
	This follows at once from Proposition \ref{prop:asymptotic-domination-via-error-function}.	
\end{proof}

Before proceeding with the theory, let us give a few simple examples of operator semigroups $\SS$ and $\T$ where asymptotic domination (and, in fact, even equivalence) occurs:

\begin{example} \label{ex:doubly-stochastic-in-dim-2}
	Let $X = \mathbb{R}^2$ be endowed with its usual order and let $A \in \R^{2\times 2}$ by given by
	$
		A =
		\begin{pmatrix}
			-1 & 1 \\
			1 & -1
		\end{pmatrix}
	$.
	For every $\alpha \ge 0$ the mapping $\SS_\alpha: [0,\infty) \to \R^{2 \times 2}$, given by
	\begin{align*}
		\SS_\alpha(t) := e^{t \alpha A} =
		\frac{1}{2}
		\begin{pmatrix}
			1 + \exp(-2\alpha t) & 1 - \exp(-2\alpha t)  \\
			1 - \exp(-2\alpha t) & 1 + \exp(-2\alpha t)
		\end{pmatrix}
	\end{align*}
	for all $t \ge 0$, defines a doubly stochastic semigroup on $\R^2$; in fact, it is not difficult to see that all doubly stochastic $C_0$-semigroups on $\R^2$ are of this form. If $\alpha > 0$, then the matrix $S_\alpha(t)$ converges to
	\begin{align*}
		P =
		\frac{1}{2}
		\begin{pmatrix}
			1 & 1 \\
			1 & 1
		\end{pmatrix}
	\end{align*}
	as $t \to \infty$. This implies that, for all values $\alpha,\beta > 0$, the functions $S_\alpha$ and $S_\beta$ are asymptotically equivalent. Yet, for $\alpha \not= \beta$, we neither have $S_\alpha(t) \ge S_\beta(t)$ nor $S_\alpha(t) \le S_\beta(t)$ for any time $t > 0$.
\end{example}

The subsequent examples constitute, in a sense, more general versions of Example~\ref{ex:doubly-stochastic-in-dim-2}.

\begin{examples} \label{exs:same-limit}
	Let $J = \N_0$ or $J = [0,\infty)$ and let $X$ be an ordered Banach space.
	
	(a) Let $\SS,\T: J \to \mathcal{L}(X)$ be two positive semigroups on $X$ for which the strong operator limits $\lim_{t \to\infty} \SS(t)$ and $\lim_{t \to \infty} \T(t)$ exist and coincide. Then $\SS$ and $\T$ are asymptotically equivalent.
	
	(b) Let $\T: J \to \mathcal{L}(X)$ be a positive semigroup on $X$ which converges strongly as $t \to \infty$ and let $\alpha,\beta > 0$. Let us define $\T_\alpha(t) = \T(\alpha t)$ and $\T_\beta(t) = \T(\beta t)$ for all $t \in J$. Since the strong operator limits of $\T_\alpha(t)$ and $\T_\beta(t)$ for $t \to \infty$ exist and coincide, it follows from~(a) that $\T_\alpha$ and $T_\beta$ are asymptotically equivalent.
\end{examples}

In Section~\ref{section:long-term-behaviour-of-asymp-dom-sg} we prove that several properties concerning the long term behaviour of operator semigroups are inherited via domination. The key to many of those results is the following simple, but very useful lemma.

\begin{lemma} \label{lem:key-lemma}
	Let $J$ be a subset of $\R$ which is not bounded above, let $X$ be an ordered Banach space and consider two mappings $f,g: J \to X_+$. If $f \asd g$, then we can find a strictly increasing sequence $(t_k)_{k \in \N_0} \subseteq J$ converging to $\infty$ and a vector $q \in X_+$ such that $f(t_k) \le g(t_k) + q$ for all $k \in \N_0$.
\end{lemma}
\begin{proof}
	Since $f \asd g$, Proposition~\ref{prop:asymptotic-domination-via-error-function} yields a mapping $r: J \to X_+$ which fulfils $\lim_{t \to \infty} r(t) = 0$ and $f(t) \le g(t) + r(t)$ for all $t \in J$. Now, choose a strictly increasing sequence $(t_k)_{k \in \N_0} \subseteq J$ which converges to $\infty$ and which fulfils $\|r(t_k)\| \le \frac{1}{2^k}$ for all $k \in \N_0$. Setting $q := \sum_{k \in \N_0} r(t_k)$, we obtain the assertion.
\end{proof}

We point out that the terminology of \emph{asymptotic domination} can, of course, also be generalised to directed sets $J$ instead of subsets of $\R$. However, the above Lemma~\ref{lem:key-lemma} generalises to this setting only if there exists a cofinal countable set in $J$. The following consequence of Lemma~\ref{lem:key-lemma} will be very useful in section \ref{section:long-term-behaviour-of-asymp-dom-sg}.

\begin{corollary} \label{cor:existence-of-cluster-points}
	Let $X$ be an ordered Banach space and let $f: \N_0 \to X_+$ be asymptotically dominated by a constant vector $x \in X_+$.
	\begin{enumerate}[label=(\alph*)]
		\item If every order interval in $X$ is weakly compact, then the sequence $(f(n))_{n \in \N_0}$ has a weakly convergent subsequence.
		\item If every order interval in $X$ is compact, then the sequence $(f(n))_{n \in \N_0}$ has a norm convergent subsequence.
	\end{enumerate}
\end{corollary}
\begin{proof}
	According to Lemma~\ref{lem:key-lemma} we can find a subsequence $(f(n_k))_{k \in \N_0}$ of\\ $(f(n))_{n \in \N_0}$ which is contained in an order interval in $X$. This clearly implies assertion~(b), and it also implies assertion~(a) due to the Eberlein--\v{S}mulian theorem \cite[Theorem~V.6.1]{dunford1958} or \cite[Corollary 2 in Paragraph~IV.11.1]{Schaefer66}).
\end{proof}

\section{Long term behaviour of asymptotically dominated semigroups} \label{section:long-term-behaviour-of-asymp-dom-sg}

\subsection{Almost periodicity via asymptotic domination}

Consider two positive operator semigroups $\T$ and $\SS$ on an ordered Banach space $X$ and assume that $\T$ asymptotically dominates $\SS$. This subsection is concerned with the question whether almost periodicity of the dominating semigroup $\T$ is inherited by the dominated semigroup $\SS$. In case that $X$ is a Banach lattice with order continuous norm the answer to this question is ``yes'' (under appropriate technical assumptions) as shown in \cite[Theorem~4.3]{Emelyanov2001}.

If $X$ is  a general ordered Banach space, say with generating and normal cone, than the situation seems to be more involved. One can, for instance, consider so-called \emph{uniformly order convex} spaces (cf.\ \cite[Definition~2.1]{Emelyanov2001}). This class contains the spaces with additive norm on the positive cone and the spaces whose norm is uniformly convex (cf.\ \cite[Example~2.2]{Emelyanov2001}). It was shown in \cite[Corollary~3.4]{Emelyanov2001} that, on such spaces, a version of almost periodicity (so-called strong asymptotic compactness, cf.\ \cite[the beginning of Section~3]{Emelyanov2001}) is inherited by asymotically dominated semigroups in case that the dominating semigroup is contractive. The question thus arises what happens if the dominating semigroup is not contractive but only bounded or if the underlying Banach space is not uniformly order convex. A result in this direction was given in \cite[Theorem~15]{Emelyanov2003} in the following setting: the space $X$ is now assumed to be \emph{ideally ordered}, meaning that every order interval in $X$ is weakly compact and that the positive cone $X_+$ is  \emph{strongly normal}, i.e.\ that the mapping $X_+ \times X_+ \ni (x,y) \mapsto \dist([0,x],[0,y])$ is continous (mere \emph{normality} of the cone is equivalent to the continuity of this mapping at the point $(0,0)$). Examples of ideally ordered Banach spaces include preduals of von Neumann algebras and the space of all self-adjoint compact linear operators on a given Hilbert space (cf.\ \cite[p.\ 12]{Emelyanov2003}); in fact, all order intervals in the latter space are even compact. If a positive, time-discrete and almost periodic semigroup $\T$ on an ideally ordered Banach space asymptotically dominates a positive semigroup $\SS$, then it is shown in \cite[Theorem~15]{Emelyanov2003} that $\SS$ is at least weakly almost periodic; if all order intervals $X$ are even norm compact, then $\SS$ is even almost periodic.

In the following theorem we generalise this result in that we remove the condition of strong normality of the positive cone.

\begin{theorem} \label{thm:almost-periodicity-via-asymptotic-domination}
	Let $J = \N_0$ or $J = [0,\infty)$ and let $X$ be an ordered Banach space with generating cone. Consider two positive semigroups $\SS, \T: J \to \mathcal{L}(X)_+$ and assume that $\T$ asymptotically dominates $\SS$. In case that $J = [0,\infty)$, assume in addition that $\SS$ is a $C_0$-semigroup.
	\begin{enumerate}[label=(\alph*)]
		\item If $\T$ is almost periodic and all order intervals in $X$ are weakly compact, then $\SS$ is weakly almost periodic.
		\item If $\T$ is almost periodic and all order intervals in $X$ are compact, then $\SS$ is almost periodic.
	\end{enumerate}
\end{theorem}

Note that the assumption that all order intervals be weakly compact is for instance fulfilled in every reflexive Banach space ordered by a normal cone; an important example for such a space is the selfadjoint part of a noncommutative $L^p$-space ($1 < p < \infty$) with respect to a semi-finite trace. Other examples are AL-Banach lattices and preduals of von Neumann algebra since they are projection bands in their bi-dual, see Proposition \ref{prop:band-in-bidual-implies-weakly-compact-intervals}.

\begin{proof}[Proof of Theorem~\ref{thm:almost-periodicity-via-asymptotic-domination}]
	Let $x \in X_+$ be given and consider an arbitrary sequence $(y_n)_{n \in \N_0}=(S_{t_n}x)_n \in \N_0$ in the orbit $\{S_t x|\, t \in J\}$. We show that a subsequence of $(y_n)_{n \in \N_0}$ converges with respect to the weak topology in case~(a) and with respect to the norm topology in case~(b).
	
	 If $(t_n)_{n \in \N_0}$ is bounded, it possesses a norm convergent subsequence since the set $\{S_t x| \, t \in [0,t_0] \cap J\}$ is compact  for every $t_0 \ge 0$ (here we need strong continuity of $\SS$ in case that $J = [0,\infty)$). So assume that the sequence $(t_n)_{n \in \N_0}$ is unbounded. After passing to a subsequence, we may then assume that $t_n \to \infty$. Passing once again to a subsequence, we may thus assume that the sequence $(\T_{t_n} x)_{n \in \N_0}$ is norm convergent in $X$. Since the mapping $n \mapsto S_{t_n}x$ is asymptotically dominated by the mapping $n \mapsto T_{t_n}x$, we conclude that $n \mapsto S_{t_n}x$ is in fact asymptotically dominated by the constant vector $y := \lim_{n \to \infty} T_{t_n}x$. Hence, the assertion follows from Corollary~\ref{cor:existence-of-cluster-points}.
\end{proof}

As is well known, every weakly almost periodic semigroup admits a so-called \emph{Jacobs--deLeeuw--Glicksberg} decomposition; see for instance \cite[Section~2.4]{Krengel1985}, \cite[Section~V.2]{Engel2000} or \cite[Section~16.3]{Eisner2015} where this is explained in great detail. In the situation of the above theorem we obtain additional information about this decomposition:

\begin{corollary}
	Suppose that we are in the situation of Theorem~\ref{thm:almost-periodicity-via-asymptotic-domination} and that the assumptions of~(a) are fulfilled. Denote by $P_{\T}$ and $P_{\SS}$ the projections onto the spaces of reversible elements for $\T$ and $\SS$ which are given by the Jacobs--de Leeuw--Glicksberg decomposition. Then $P_{\T} \ge P_{\SS} \ge 0$; in particular, $\operatorname{rank} P_{\T} \ge \operatorname{rank} P_{\SS}$.
\end{corollary}
\begin{proof}
	It follows from Theorem~\ref{thm:almost-periodicity-via-asymptotic-domination} that $\SS$ is weakly almost periodic and thus admissible for the Jacobs--de Leeuw--Glicksberg decomposition. The construction of this decomposition together with the asymptotic domination implies that $P_{\T} \ge P_{\SS}$. Moreover, $P_{\SS} \ge 0$ by the positivity of $\SS$. Finally, the rank estimate follows from \cite[Proposition~2.1.3]{Emelyanov2007}.
\end{proof}

\subsection{Mean ergodic theorems}

In this subsection we prove mean ergodic theorems for positive operators on ordered Banach spaces with weakly compact order intervals. The assertion of our first result, Theorem~\ref{thm:power-of-mean-ergodic-operator-is-mean-ergodic}, is not directly related to asymptotic domination; however, the concept of asymptotic domination occurs in its proof.

Recall that it might happen that a positive operator $T$ on a Banach lattice is mean ergodic, while not all the powers of $T$ are mean ergodic. Indeed, Sine \cite{Sine1976} constructed an example of a Koopman operator $T$ on the space of continuous functions on a certain compact Hausdorff space such that $T$ is mean ergodic, while $T^2$ is not; further counterexamples can be found in \cite{Emelyanov2009} and in \cite{GerlachWAP}.

However, the situation changes if we require additional regularity properties from the space $X$. If, for instance, $X$ is a Banach lattice with order continuous norm and $T \in \L(E)$ is positive and mean ergodic, then $T^r$ is indeed mean ergodic for every $r \in \N$. This was proved by Derriennic and Krengel in \cite[Proposition~4.5]{Derriennic1981}. In \cite[Theorem~12]{Emelyanov2003} this result was generalised to ideally ordered Banach spaces. In the next theorem we show that it is actually true on every ordered Banach space with weakly compact order intervals.

\begin{theorem} \label{thm:power-of-mean-ergodic-operator-is-mean-ergodic}
	Let $X$ be an ordered Banach space with generating cone and assume that all order intervals in $X$ are weakly compact. Let $T\in \L(X)$ be a mean ergodic positive operator and let $r \in \N$. Then $T^r$ is mean ergodic, too.
\end{theorem}
\begin{proof}
	First note that the positive cone $X_+$ is normal since every order interval is weakly compact. Define $\T := (T^n)_{n \in \N_0}$. Since $T$ is mean ergodic, the sequence $(A_n(\T))_{n \in \N}$ converges strongly to a positive projection $P$. Moreover, $\lim_{n \to \infty} n^{-1}T^n = 0$ holds with respect to the strong operator topology. The Ces{\`a}ro means of the semigroup $\SS := (T^{rn})_{n \in \N_0}$ fulfil the estimate
	\begin{align*}
		A_n(\SS) =\frac{1}{n} \sum_{k=0}^{n-1}T^{rk} \le \frac{r}{rn}\sum_{k=0}^{rn-1} T^k = r \cdot A_{rn}(\T)
	\end{align*}
	for each $n \in \N$. This implies that the operator $T^r$ is Ces{\`a}ro bounded (since the cone $X_+$ is normal and generating) and that the sequence $(A_n(\SS))_{n \in \N}$ is asymptotically dominated by the operator $r \cdot P$.
	
	Hence, it follows from Corollary~\ref{cor:existence-of-cluster-points} that, for every $x \ge 0$, the sequence $(A_n(\SS)x)_{n\in \N}$ has a weakly convergent subsequence. Since $n^{-1}T^{rn} = \frac{r}{rn}T^{rn}$ we conclude the $n^{-1}T^{rn}$ converges strongly to $0$ as $n \to \infty$, so the assertion follows from Eberlein's ergodic theorem (see \cite[Theorem~1.1.7]{Emelyanov2007} or \cite[Theorem~8.20]{Eisner2015}).
\end{proof}

The remaing results in this section deal with the question how mean ergodicity interacts with asymptotic domination. To prove those results, we need the following elementary lemma.

\begin{lemma} \label{lem:asymptotic-domination-carries-over-to-cesaro-means}
	Let $X$ be an ordered Banach space with generating cone and assume that all order intervals in $X$ are weakly compact. Let $J = \N_0$ or $J = [0,\infty)$ and consider two mappings $f, g: J \mapsto X_+$ which satisfy $g\asd f$.
	
	\begin{enumerate}
		\item[(a)] Let $J = \N_0$. Then the Ces{\`a}ro means of $g$ are asymptotically dominated by the Ces{\`a}ro means of $f$, i.e.\ we have
			\begin{align*}
				\left(\frac{1}{n}\sum_{k=0}^{n-1}g(k)\right)_{n \in \N} \asd \left(\frac{1}{n}\sum_{k=0}^{n-1}f(k)\right)_{n \in \N}.
			\end{align*}
		\item[(b)] Let $J = [0,\infty)$ and assume in addition that $f$ and $g$ are continuous. Then the Ces{\`a}ro means of $g$ are asymptotically dominated by the Ces{\`a}ro means of $f$, i.e.\ we have
			\begin{align*}
				\left(\frac{1}{t}\int_0^t g(s)ds\right)_{t \in (0,\infty)} \asd \left(\frac{1}{t}\int_0^t f(s)ds\right)_{t \in (0,\infty)}.
			\end{align*}
	\end{enumerate}
\end{lemma}
\begin{proof}
	The case of $J = \N_0$ follows immediately from the subadditivity of the functional $\distPos$, so assume that $J = [0,\infty)$. Then the subadditivity together with the continuity of $\distPos$ implies
	\begin{align*}
		0 & \le \distPos\left(\frac{1}{t} \int_0^t f(s)\dx s - \frac{1}{t} \int_0^t g(s) \dx s \right)\\
		  & = \distPos\left(\frac{1}{t}\int_0^t(f(s)-g(s)) \dx s \right) \le \frac{1}{t}\int_0^t\distPos(f(s)x -g(s)) \dx s.
	\end{align*}
	for each $t \in (0,\infty)$. This yields the assertion since $\lim_{s \to \infty}\distPos(f(s) - g(s)) = 0$ by assumption.
\end{proof}

Consider a positive operator $T$ on a Banach lattice $E$ and assume that $T$ is \emph{power order bounded}, meaning that, for each $x \in E$, the orbit $\{T^nx: \; n \in \N_0\}$ is contained in an order interval of $E$. If $E$ has order continuous norm, then it is not difficult to see that $T$ is mean ergodic. Next, we generalise this observation in several respects.

\begin{definition}
	Let $X$ be an ordered Banach space and let $J = \N_0$ or $J = [0,\infty)$. A positive semigroup $\T = (T_t)_{t \in J}$ on $E$ is called  {\em asymptotically order bounded} if, for every $x \in X_+$, there exists a vector $y \in X_+$ which asymptotically dominates the orbit mapping $J \ni t \to T_tx \in X$.
\end{definition}

Observe that, if $X_+$ is generating and normal, every positive semigroup on $X$ which is asymptotically ordered bounded is automatically bounded due to the uniform boundedness principle.

\begin{proposition} \label{prop:asymptotically-power-order-bounded-operators}
	Let $X$ be an ordered Banach space with generating cone and let $J = \N_0$ oder $J = [0,\infty)$. Let $\T = (T_t)_{t \in J}$ be a positive semigroup which is asymptotically order bounded and assume that $\T$ is strongly continuous in case that $J = [0,\infty)$. If all order intervals in $X$ are weakly compact, then $\T$ is mean ergodic.
\end{proposition}
\begin{proof}
	Since every order intervall in $X$ is weakly compact, it follows that the positive cone on $X$ is normal and hence, the semigroup $\T$ is bounded.
	
	Let $x \in X_+$. By assumption there exists a vector $y \in X_+$ which asymptotically dominates the orbit $(T_tx)_{t \in J}$. Then $(A_t(\T)x)_{t \in J \setminus \{0\}}$ is asymptotically dominated by $y$ according to Lemma~\ref{lem:asymptotic-domination-carries-over-to-cesaro-means}. Corollary \ref{cor:existence-of-cluster-points}(a) and Eberlein's ergodic theorem (see \cite[Theorem~1.1.7]{Emelyanov2007}, or \cite[Theorem~8.20]{Eisner2015} and \cite[Theorem~V.4.5(c)]{Engel2000}) yield the assertion.
\end{proof}

\begin{open_problem}
	In \cite[Theorem~1.1]{Emelyanov1999}, E.~Yu.~Emel'yanov and the second author proved a converse result: if every (asymptotically) order bounded positive operator on a {\em Banach lattice} is mean ergodic, then every order interval is weakly compact. It is an open problem whether this assertion holds on more general ordered Banach spaces.
\end{open_problem}

Finally, we consider the question whether mean ergodicity of a semigroup $\T$ is inherited by a semigroup $\SS$ which is asymptotically dominated by $\T$. Previous results on this question were given by Arendt and Batty on Banach lattices with order continuous norm \cite[Theorem~1.1]{Arendt1992} (where they assumed domination instead of asymptotic domination) and by Emel'yanov and the second author for time discrete semigroups ($J = \N_0$) on an ordered Banach space with strongly normal cone and weakly compact order intervals \cite[Theorem~14]{Emelyanov2003}; see also \cite[Theorem~2.1.11]{Emelyanov2007}. A similar result was proved for dominated $C_0$-semigroups on the same kind of ordered Banach spaces \cite[Theorem 2.1.12]{Emelyanov2007}. The following theorem shows that neither strong  normality of the cone nor the restriction to purely dominated semigroups is  needed for this result.

\begin{theorem} \label{thm:asymptotic-domination-and-mean-ergodicity}
	Let $X$ be an ordered Banach space with generating cone and assume that all order intervals in $X$ are weakly compact. Let $J = \N_0$ or $J = [0,\infty)$ and consider two positive semigroups $\SS = (S_t)_{t \in J}$ and $\T = (T_t)_{t \in J}$ on $X$ which we assume to be strongly continuous in case that $J = [0,\infty)$. Assume that $\T$ is mean ergodic and that it asymptotically dominates $\SS$.
	\begin{enumerate}
		\item[(a)] If $J = \N_0$, then $\SS$ is mean ergodic.
		\item[(b)] Let $J = [0,\infty)$. If $(T_t/t)_{t \in [1,\infty)}$ is bounded, then $\SS$ is mean ergodic.
		\item[(c)] In both cases~(a) and~(b) the corresponding mean ergodic projections $P_\SS$ and $P_\T$ satisfy $0 \le P_\SS \le P_\T$ and hence $\operatorname{rank} P_\SS \le \operatorname{rank} P_\T$.
	\end{enumerate}
\end{theorem}

For the proof we use the following observation which can be checked by a few straightforward computations: if $\T = (T(t))_{t \in [0,\infty)}$ is a mean ergodic $C_0$-semigroup on a Banach space $X$, then
\begin{align}
	\label{eq:conv-mean-ergodic}
	1/t \cdot T(t) \int_0^s T(\tau)x \, d\tau \to 0 \quad \text{as} \quad t \to \infty
\end{align}
for all $x \in X$ and all $s > 0$.

\begin{proof}[Proof of Theorem~\ref{thm:asymptotic-domination-and-mean-ergodicity}]
	First note that it follows from the uniform boundedness principle that the cone $X_+$ in $X$ is normal. Fix $x \in X_+$. It follows from Lemma~\ref{lem:asymptotic-domination-carries-over-to-cesaro-means} and Corollary~\ref{cor:existence-of-cluster-points}(a) that the net $(A_t(\SS)x)_{t \in J \setminus \{0\}}$ possesses a weakly convergent subsequence.
	
	(a) Let $J = \N_0$. Since $\T$ is ergodic, $T_n/n$ converges strongly to $0$ as $n \to \infty$. As $(S_n x)_{n \in \N_0}$ is asymptotically dominated by $(T_n x)_{n \in \N_0}$, it follows that $(S_n/n) x \to 0$ as $n \to \infty$ and that $\SS$ is Ces{\`a}ro bounded (by Lemma~\ref{lem:asymptotic-domination-carries-over-to-cesaro-means}), so the assertion follows from Eberlein's Ergodic theorem \cite[Theorem~1.1.7]{Emelyanov2007}.
	
	(b) Let $J = [0,\infty)$. It follows from~\eqref{eq:conv-mean-ergodic} that the net $(T_t/t)_{t \in [1,\infty)}$ converges pointwise to $0$ on a dense subset of $X$. Since the net is bounded by assumption, it thus converges strongly to $0$ on the entire space $X$. Hence, $(S_t/t)_{t \in [1,\infty)}$ converges strongly to $0$, too; since $\SS$ is Ces{\`a}ro bounded by Lemma~\ref{lem:asymptotic-domination-carries-over-to-cesaro-means}, the assertion again follows from Eberlein's ergodic theorem~\cite[Theorem~1.1.7]{Emelyanov2007}.
	
	(c) Clearly, $0 \le P_\SS \le P_\T$. The rank estimate thus follows from~\cite[Proposition~2.1.3]{Emelyanov2007}.
\end{proof}

In part~(b) of the above theorem, the assumption that $(T_t/t)_{t \in (0,\infty)}$ be bounded follows automatically from the mean ergodicity of $\T$ in case that the norm on $X$ is additive on the positive cone (see Subsection~\ref{subsection:spaces-with-additive-norm} and the appendix for details about such spaces). Indeed, we have the following result:

\begin{proposition} \label{prop:mean-ergodicity-and-additive-norm}
	Let $X$ be an ordered Banach space with generating cone and assume that the norm is additive on the positive cone, meaning that $\|x+y\| = \|x\| + \|y\|$ for all $x,y \in X_+$. Let $\T = (T_t)_{t \in [0,\infty)}$ be a positive and mean ergodic $C_0$-semigroup on $X$. Then $(T_t/t)_{t \in [1,\infty)}$ is bounded (and hence, $T_t/t \to 0$ as $t \to \infty$ with respect to the strong operator topology).
\end{proposition}
\begin{proof}
	It suffices to prove that $\sup_{t \ge 1} 1/t \cdot \lVert T(t)x\rVert < \infty$ for all $x \in X_+$, so assume for a contradiction that there exists a vector $x \in X_+$ for which we have $\sup_{t \ge 1} 1/t \cdot \lVert T(t)x \rVert = \infty$.

	We define $M := \sup_{t \in [0,1]} \lVert T(t) \rVert < \infty$. By applying~\eqref{eq:conv-mean-ergodic} for $s = 1$ we can find a number $t_0 > 0$ such that
	\begin{align}
		\label{eq:mean-ergodicity-and-additive-norm-auxiliary}
		1/t \cdot \lVert \int_t^{t+1}T(\tau) x \,d\tau \rVert = 1/t \cdot \lVert T(t) \int_0^1 T(\tau)x \, d\tau \rVert \le 1
	\end{align}
	for all $t > t_0$.

	Since $\sup_{t \in [1,t_0+1]} 1/t \cdot \lVert T(t)x\rVert < \infty$, it follows from $\sup_{t \ge 1} 1/t \lVert T(t)x \rVert = \infty$ that there exists a time $t_1 > t_0+1$ for which we have $1/t_1 \cdot \lVert T(t_1)x \rVert > M$.
	
	We now set $t := t_1-1 > t_0$ and apply~\eqref{eq:mean-ergodicity-and-additive-norm-auxiliary}. Using that the norm is additive on $X_+$, we obtain
	\begin{align*}
		\int_{t_1-1}^{t_1} 1/{t_1} \cdot \lVert T(\tau)x \rVert \,d \tau \le 1/t \int_t^{t+1} \lVert T(\tau)x \rVert \,d \tau = \frac{1}{t} \lVert \int_t^{t+1} T(\tau) x \, d\tau \rVert \le 1.
	\end{align*}
	Hence, there exists a number $\tau \in [t_1-1,t_1]$ such that $1/t_1 \cdot \lVert T(\tau)x\rVert \le 1$ and thus,
	\begin{align*}
		1/t_1 \cdot \lVert T(t_1)x \rVert \le 1/t_1 \cdot \lVert T(t_1-\tau) \rVert \cdot \lVert T(\tau)x \rVert \le M.
	\end{align*}
	This is a contradiction.
\end{proof}

\section{Lower bound theorems} \label{section:lower-bound-theorems}

\subsection{Introduction}

In the previous chapter we proved theorems about the asymptotic behaviour of positive semigroups under  the assumption that they are asymptotically dominated by another semigroup. In the present chapter we change our perspective and assume instead that the semigroup orbits asymptotically dominate a time-independent vector, called \emph{lower bound},  and conclude that roughly spoken the semigroup converges. More precisely we make the following definition:

\begin{definition} \label{strong stability}
	Let $J = \N_0$ or $J = [0,\infty)$. A semigroup $\T = (T_t)_{t \in J}$ on a Banach space $X$ is called \emph{strongly convergent} if, for every $x \in X$, the limit $\lim_{t \to \infty}T_t x$ exists with respect to the norm topology on $X$.
\end{definition}

\begin{rem}\label{stabilityprojection}
	If a semigroup $\T = (T_t)_{t \in J}$ is strongly convergent, then its limit operator is always a projection which commutes with all operators $T_t$ and which we denote by $P_\T$.
\end{rem}

It was first observed by Lasota and Yorke that the existence of non-zero lower bounds implies strong stability for Markov semigroups  on $L^1$-spaces; see \cite[Theorem~2 and Remark~3]{Lasota1982} and \cite[Theorem~1.1]{Lasota1983}. Later on, this result was generalised to more general semigroups on $L^1$-spaces, see \cite[Theorem~2.1]{Zalewska-Mitura1994}, \cite[Corollary~5.1]{Komornik1993} and \cite[Section~3]{GerlachLB}. Convergence theorems in case that there exists a different lower bound for each orbit can be found in \cite[Theorem~1.1]{Ding2003} and \cite[Sections~4 and~5]{GerlachLB}.

We note that all the results quoted above are formulated in the setting of $L^1$-spaces (or, from a more abstract point of view, on AL-Banach lattices). The reason is that, on a Banach lattice $X$ and under mild technical assumptions, the existence of an operator which possesses a lower bound implies that the space is an AL-space with respect to an equivalent norm \cite[Theorem~3.7]{GerlachLB}. There are, however, some results available in the more general setting of non-commutative $L^1$-spaces, see \cite{Ajupov1987, Sarymsakov1987, Emelyanov2006}; we refer also to \cite[Section~3.3]{Emelyanov2007} for an overview of the latter topic.

In Subsection \ref{subsection:stability and lower bounds} we prove convergence theorems for semigroups possessing a lower bound on general ordered Banach space whose norm is additive on the positive cone. We start with a brief discussion of those spaces in Subsection~\ref{subsection:spaces-with-additive-norm} (see also the appendix where several non-trivial examples of such spaces are presented). In the subsequent two subsections we present our convergence results.

It is worthwhile noting that there is a lower bound type theorem of Soca{\l}a \cite[Theorem~1]{Socala2002} that also works on general ordered Banach spaces, but whose assumptions differ considerably from ours.

\subsection{AN--spaces (Spaces with additive norm)} \label{subsection:spaces-with-additive-norm}

The norm on an ordered Banach space $X$ is said to be \emph{additive on the positive cone} -- or, for short, \emph{additive on $X_+$} -- if $\|x+y\| = \|x\| + \|y\|$ for all $x,y \in X_+$.

\begin{definition}
	We call an ordered Banach space $(X,X_+)$ an \emph{AN--space} if the positive cone $X_+$ is generating and if the norm is additive on $X_+$.
\end{definition}

Note that the positive cone of an AN--space is \emph{normal} since an additive norm is monotone. \smallskip

A Banach lattice whose norm is additive on the positive cone is called an \emph{AL-Banach lattice} or, for short, an \emph{AL-space}. Every $L^1$-space over an arbitrary measure space is an AL-space and conversely, every AL-space is (isometrically lattice) isomorphic to an $L^1$-space. Typical examples of AN--spaces which are not  AL--Banach lattices  are preduals of von Neumann algebras or noncommutative $L^1$--spaces. Several further examples are presented in Appendix~\ref{appendix:example-additive-norm}.

Let $X$ be an AN--space. We note that $X$ is called a \emph{base norm space} if it fulfils, in addition, the equation
\begin{align} \label{eq:base-norm-formula}
	\|x\| = \inf \{\|y\| + \|z\|: \; y,z \in X_+, \; x = y-z\}
\end{align}
for all $x \in X$. We can always endow a given AN--space $X$ with an equivalent norm $\|\cdot\|_1$ which agrees with the given norm on $X_+$ (and hence is still additive on $X_+$) and which renders $X$ a base norm space; more precisely, we have the following proposition:

\begin{proposition} \label{prop:renorming-to-become-a-base-norm-space}
	Let $X$ be an AN--space. Then there exists an equivalent norm $\|\cdot\|_1$ on $X$ which possesses the following properties:
	\begin{enumerate}[label=(\alph*)]
		\item On $X_+$, the norm $\|\cdot\|_1$ coincides with the original norm.
\item The space $(X,\|\cdot\|_1)$ is even a base norm space.
	\end{enumerate}
\end{proposition}
\begin{proof}
	Simply define $\|x\|_1$ to be the right hand side of~\eqref{eq:base-norm-formula} for every $x \in X$. Then it is easy to check that $\|\cdot\|_1$ is an equivalent norm on $X$ which fulfils both properties~(a) and (b).
\end{proof}

For more information about base norm space norm spaces we refer, for instance, to \cite{Nagel1974} and \cite[Section~1.4]{Batty1984}.

Let $X$ be an AN--space. Then there exists a uniquely determined functional $\one \in X'_+$ that fulfils $\langle \one, x \rangle = \|x\|$ for all $x \in X_+$; we call $\one$ the \emph{norm functional} on $X$. Let us collect a few properties of the norm functional in the following proposition.

\begin{proposition} \label{prop:properties-of-norm-functional}
	Let $X$ be an AN--space. Then the norm functional $\one \in X'_+$ has the following properties:
	\begin{enumerate}[label=(\alph*)]
		\item If $X \not= \{0\}$, then $\|\one\| = 1$.
		\item The unit ball of $X'$ is contained in $[-\one,\one]$.
		\item We have $\|x''\| = \langle x'', \one \rangle$ for all $x'' \in X''_+$.
	\end{enumerate}
\end{proposition}
\begin{proof}
	(a) The inequality $\|\one\| \ge 1$ is obvious as $X \not= \{0\}$ and hence $X_+ \not= \{0\}$. On the other hand, we can decompose every $x \in X$ as $x = y-z$ for some $y,z \in X_+$ and it thus follows from the inverse triangle inequality that
	\begin{align*}
		\|x\| \ge |\|y\|-\|z\|| = |\langle \one, y \rangle - \langle \one, z\rangle| = |\langle \one, x\rangle|.
	\end{align*}
	(b) Whenever $x' \in X'$ is a functional of norm $\|x'\| \le 1$ we have $\langle \one - x', x \rangle = \|x\| - \langle x',x\rangle \ge 0$ for all $x \in X_+$, so $\one \ge x'$. By applying this to $-x'$, we also obtain $-\one \le x'$.
	
	(c) The assertion is obvious if $X = \{0\}$, so assume $X \not= \{0\}$. Let $x'' \in X''_+$. Since $\|\one\| = 1$, it follows that $\|x''\| \ge \langle x'',\one \rangle$. If, on the other hand, $x' \in X'$ has norm at most $1$, then it follows from~(b) that $-\one \le x' \le \one$. Thus,
	\begin{align*}
		-\langle x'', \one \rangle \le \langle x'', x' \rangle \le \langle x'', \one \rangle,
	\end{align*}
	so $\langle x'', \one \rangle \ge |\langle x'', x'\rangle|$.
\end{proof}

The rather simple structure theory of AL-Banach lattices should not tempt us into believing that AN--spaces are similarly well-behaved. In Appendix~\ref{appendix:example-additive-norm} we present a variety of example which demonstrate that  AN--spaces may have rather disparate regularity properties. This topic is also related to regularity properties of so-called \emph{KB-spaces} which we discuss in the following remark.

\begin{rem} \label{rem:kb-for-non-lattices}
	Recall that a Banach lattice $X$ is called a \emph{KB-space} if one of the following equivalent assertions is fulfilled \cite[Proposition II.5.15]{Schaefer74}.
	\begin{enumerate}
		\item Every increasing norm bounded sequence in $X$ is norm convergent.
		\item The space $X$ is a projection band in its bidual.
	\end{enumerate}
	Note that every KB-space $X$ has order continuous norm and hence, every order interval in such a space $X$ is weakly compact. We also point out that it is easy to see that every AL-Banach lattice is a KB-space.

	Let us now discuss what happens to the equivalence ``(i) $\Leftrightarrow$ (ii)'' if $X$ is only an ordered Banach space with generating and normal cone (note that assertion~(ii) only makes sense if $X_+$ is generating and normal).
	
	The implication ``(ii) $\Rightarrow$ (i)'' remains true. Indeed, let $(x_j)$ be an increasing norm-bounded sequence (or net) in $X$. Then $(x_j)$ converges to an element $x'' \in X''$ with respect to the weak${}^*$-topology. If (ii) holds and $P \in \L(X'')$ denotes the band projection onto $X$, one immediately checks that $Px'' = x''$; hence, $x := x''$ is in fact an element of $X$ and $(x_j)$ converges weakly to $x$. This implies (i) since every increasing and weakly convergent net is automatically norm convergent \cite[Theorem~V.4.3]{Schaefer66}.
	
	By contrast, the implication ``(i) $\Rightarrow$ (ii)'' fails in general. Indeed, there exist AN--spaces for which not every order interval is weakly compact (see for instance Remark~\ref{rem:an-spaces-and-banach-lattices-appendix} in the appendix). Such a space $X$ cannot be a projection band in its bi-dual according to Proposition~\ref{prop:band-in-bidual-implies-weakly-compact-intervals}; yet, it is easy to see that every increasing norm bounded sequence in an AN-space is norm convergent.
\end{rem}

\subsection{Individual and universal lower bounds}

In the remaining part of the article we derive convergence theorems for positive operator semigroups from the existence of so-called \emph{lower bounds} which are defined as follows.

\begin{definition} \label{def:individual-lower-bounds}
	Let $X$ be an ordered Banach space and let $\T = (T_t)_{t \in J}$ be a positive one-parameter semigroup on $X$ (where $J = \N_0$ or $J = [0,\infty)$).
	\begin{enumerate}
		\item[(i)] Let $f \in X_+$ be of norm $1$. A vector $h > 0$ is called an \emph{(individual) lower bound} for the pair $(\T,f)$ if $h$ is asymptotically dominated by $(T_tf)_{t \in J}$.
		\item[(ii)] A vector $h > 0 $ is called a \emph{universal lower bound for $\T$} if it is a lower bound for the pair $(\T,f)$ for each $f \in X_+$ of norm $1$.
	\end{enumerate}
	For the rest of this definition we assume that $\T$ is a $C_0$-semigroup in case that $J = [0,\infty)$.
	\begin{enumerate}
		\item[(iii)] Let $f \in X_+$ be of norm $1$. A vector $h > 0$ is called an \emph{(individual) mean lower bound} for the pair $(\T,f)$ if $h$ is asymptotically dominated by $(A_t(\T)f)_{t \in J}$.
		\item[(iv)] A vector $h > 0$ is called a \emph{universal mean lower bound for $\T$} if it is a mean lower bound for the pair $(\T,f)$ for each $f \in X_+$ of norm $1$.
	\end{enumerate}
\end{definition}

If the semigroup $\T$ is clear from the context we say $h_f$ is a \emph{lower bound for $f$}, and so on.

Universal lower bounds on $L^1$-spaces were introduced by Lasota and Yorke in \cite{Lasota1982} in order to study the Perron--Frobenius operator associated with certain dynamical systems. They derived strong convergence of the semigroup from the existence of such a lower bound. Individual lower bounds were first considered by Ding \cite{Ding2003} for semigroups of Perron--Frobenius operators. Ding's results can for instance be employed in the study of certain mixing properties of dynamical systems, see \cite[Section~2]{GerlachPFConv}. Individual lower bounds for more general positive semigroups on AL-Banach lattices were in the focus of a recent paper \cite{GerlachLB} by Gerlach and the first of the present authors. On AL-Banach lattices mean lower bounds were for instance considered in \cite{Emelyanov2004} and on preduals of von Neumann algebras they were studied in \cite{Emelyanov2006}. In the rest of this article we will generalise several results about lower bounds of semigroups from AL-Banach lattices and preduals of von Neumann algebras to general AN--spaces.

The following remark, which shows that the existence of lower bounds is consistent with renorming arguments, will be used tacitly on several occasions in the sequel:

\begin{rem} \label{rem:renorming-allowed}
	(a) Let $X$ be an
	ordered Banach space and let $\T = (T_t)_{t \in J}$ be a positive one-parameter semigroup on $X$ (where $J = \N_0$ or $J = [0,\infty)$). Assume that, for every $f \in X_+$ of norm $1$, there exists a (mean) lower bound $h_f \in X_+$ of $(\T,f)$ such that $\inf \{\|h_f\|: \; f \in X_+, \; \|f\| = 1\} > 0$
	
	This assumptions remains true whenever we endow $X$ with an equivalent norm. The proof of this observation is elementary, so we omit it.
	
	(b) A similar observation is also true for universal (mean) lower bounds.
\end{rem}

We now recall the notion of a \emph{Markov operator} which is essential for our approach.

\begin{definition}
	Let $X$ be an AN--space. We call a linear operator $T \in \L(X)$ a \emph{Markov operator} if $T$ is positive and if $\|Tf\| = \|f\|$ for every $f \in X_+$. Equivalently, $T$ is positive and the norm functional $\one$ is a fixed point of the dual operator $T'$. A one-parameter semigroup $\T = (T_t)_{t \in J}$ (where $J = \N_0$ or $J = [0,\infty)$) is called a \emph{Markov semigroup} if every operator $T_t$ is a Markov operator.
\end{definition}

Note that a Markov semigroup on an AN--space $X$ is automatically bounded; this is a consequence of the decomposition property in~\eqref{eq:bounded-decomposition}. If $X$ is even a base norm space, then every Markov semigroup is actually contractive.

Our first goal is to show in Proposition~\ref{prop:renorming-to-markov-semigroup-very-general} that the existence of (mean) lower bounds for a semigroup $\T$ makes it possible to renorm the underlying space $X$ such that $X$ becomes a base norm space and such that $\T$ becomes a Markov semigroup. We start with the following simple lemma:

\begin{lemma} \label{lem:renorming}
	Let $X$ be an ordered Banach space with generating cone. Let $\psi$ be a positive linear functional on $X$ and define $\|f\|_\psi = \inf\{\langle\psi,g+h\rangle: g,h \in X_+, \, f = g-h\}$ for every $f \in X$.
	
	Then $\|\cdot\|_\psi$ is a seminorm which is additive on the positive cone which satisfies $\|f\|_\psi \le 2C \, \|\psi\| \, \|f\|$, where $C$ is the constant from the decomposition property in~\eqref{eq:bounded-decomposition}.
	
	If, moreover, $\inf\big\{\langle\psi,f\rangle: \, f\in X_+, \; \|f \| = 1 \big\} >0$, then the seminorm $\|\cdot\|_\psi$ is actually a norm which is equivalent to the original norm on $X$ and $(X, \|\cdot\|_\psi)$ is base--norm space.
\end{lemma}
\begin{proof}
	The proof is straightforward, so we omit the details.
\end{proof}

We will need a few facts about Banach limits which we recall next. Let either $J = \N_0$ or $J = [0,\infty)$; a Banach limit $\ell^\infty(J)$ is a positive linear functional $b \in \big(\ell^\infty(J)\big)'$ satisfying
\begin{enumerate}
	\item $\langle b,\one_J\rangle = 1$,
	\item $\langle b, S_u f \rangle = \langle b, f\rangle$ for all $f \in E$ and all $u \in J$, where $S_u$ denotes the left shift operator on $E$ given by $(S_u f)(t) = f(u+t)$ for all $f \in E$ and all $t \in J$.
\end{enumerate}
Note that there exists a Banach lattice on $\ell^\infty(J)$ since $J$ is a commutative semigroup and thus
amenable (see e.g.\ \cite[p.\ 178]{Lyubich1992}). As a consequence of the definition of a Banach limit, we have $\langle b, f\rangle = \lim_{t \to \infty}f(t)$ for all $f \in \ell^\infty(J)$ for which this limit exists.

\begin{proposition} \label{prop:renorming-to-markov-semigroup-very-general}
	Let $X$ be an ordered Banach space with generating cone and let $\T = (T_t)_{t \in J}$ be a bounded semigroup of positive operators on $X$, where either $J = \N_0$ or $J = [0,\infty)$. Assume that at least one of the following two conditions is satisfied:
	\begin{enumerate}
		\item[(a)] For every $f \ge 0$ of norm $1$, there exists a lower bound $h_f$ of $(\T,f)$ and we can find a positive functional $\varphi \in X'$ such that $\inf\{\langle \phi,h_f\rangle: \, f \in X_+, \; \|f\| = 1\} =: \beta >0$.
		\item[(b)] Either $J = \N_0$ or $J = [0,\infty)$ and $\T$ is strongly continuous. Moreover, for every $f \ge 0$ of norm $1$, there exists a mean lower bound $h_f$ of $(\T,f)$ and we can find a positive functional $\varphi \in X'$ such that $\inf\{\langle \phi,h_f\rangle: \, f \in X_+, \; \|f\| = 1\} =: \gamma >0$.
	\end{enumerate}
	Then there exists an equivalent additive norm $\|\cdot\|_1$ on $X$ such that $(X, \|\cdot\|_1)$ is a base norm space and such that all operators $T_t$ are Markov operators.
\end{proposition}
\begin{proof}
	(b) Fix a Banach limit $b \in \ell^\infty(J)'$. We define a positive functional $\psi \in X'$ by means of the formula
	\begin{align*}
		\langle\psi,g\rangle = \Big\langle b, \big(\langle\varphi,A_t(\T)g\rangle\big)_{t \in J}  \Big\rangle \qquad \forall g \in X.
	\end{align*}
	Then it is easy to see that $\psi$ is invariant under $\T$, i.e.\ we have $T_t'\psi = \psi$ for all $t \in J$. Moreover, it follows fro the assumption on $\psi$ and from the definition of a mean lower bound that $\langle \psi,f \rangle \ge \gamma$ for each $f \in X_+$ of norm $1$. Hence, the seminorm $\|\cdot\|_\psi$ from Lemma~\ref{lem:renorming} is, according to the lemma, a an equivalent norm on $X$ and $(X,\|\cdot\|_\psi)$ is a base norm space. It follows from the $\T$-invariance of $\psi$ that the semigroup $\T$ is Markov with respect to $\|\cdot\|_\psi$.
	
	(a) If $J = \N_0$ or if $J = [0,\infty)$ and $\T$ is strongly continuous, every lower bound of $(\T,f)$ is also a mean lower bound of $(\T,f)$; thus, the assertion follows from~(b) in those cases. If $J = [0,\infty)$ and $\T$ is not strongly continuous, we can argue similarly as in the proof of~(b); we only have to replace the vector $\big(\langle\varphi,A_t(\T)g\rangle\big)_{t \in J}$ in the definition of $\psi$ with $\big(\langle\varphi,T_t g\rangle\big)_{t \in J}$.
\end{proof}

\begin{corollary}  \label{cor:ind-lower-bouds-renorming}
	Let $X$ be an AN-space and let $\T = (T_t)_{t \in J}$ be a bounded positive one-parameter semigroup on $X$ (where $J = \N_0$ or $J = [0,\infty)$). Assume that at least one of the following two conditions is satisfied:
	\begin{enumerate}
		\item[(a)] For every $f \ge 0$ of norm $1$, there exists a lower bound $h_f$ of $(\T,f)$ such that $\inf\{\|h_f\|: \, f \in X_+, \; \|f\| = 1\} =: \beta >0$.
		\item[(b)] Either $J = \N_0$ or $J = [0,\infty)$ and $\T$ is strongly continuous. Moreover, for every $f \ge 0$ of norm $1$, there exists a mean lower bound $h_f$ of $(\T,f)$ such that $\inf\{\|h_f\|: \, f \in X_+, \; \|f\| = 1\} =: \gamma >0$.
	\end{enumerate}
	Then there exists an equivalent additive norm $\|\cdot\|_1$ on $X$ such that $(X, \|\cdot\|_1)$ is a base norm space and such that all operators $T_t$ are Markov operators.
\end{corollary}
\begin{proof}
	Apply the preceding Proposition~\ref{prop:renorming-to-markov-semigroup-very-general} to $\varphi = \one$ where $\one$ denotes the norm functional on $X$.
\end{proof}

\begin{corollary} \label{cor:unif-mean-lower-bounds-renorming}
	Let $X$ be an ordered Banach space with generating cone and let $\T = (T_t)_{t \in J}$ be a bounded positive one-parameter semigroup on $X$ (where $J = \N_0$ or $J = [0,\infty)$). Assume that at least one of the following two conditions is satisfied:
	\begin{enumerate}
		\item[(a)] There exists a universal lower bound $h > 0$ of $\T$.
		\item[(b)] Either $J = \N_0$ or $J = [0,\infty)$ and $\T$ is strongly continuous. Moreover, there exists a universal mean lower bound $h > 0$ of $\T$.
	\end{enumerate}
	Then there exists an equivalent additive norm $\|\cdot\|_1$ on $X$ such that $(X, \|\cdot\|_1)$ is a base norm space and such that all operators $T_t$ are Markov operators.
\end{corollary}
\begin{proof}
	According to Corollary~\ref{cor:positive-non-annihilating-functional}(b) there exists a positive linear functional $\varphi \in X'$ such that $\varphi(h) > 0$. Thus, we can apply Proposition~\ref{prop:renorming-to-markov-semigroup-very-general} to this $\phi$.
\end{proof}

In case that $X$ is a Banach lattice, the above results were implicitly contained in the proofs of Corollaries~3.5 and~4.5 of \cite{GerlachLB}.
\medskip

\subsection{Individual lower bounds and strong convergence} \label{subsection:stability and lower bounds}

Let us now proceed towards our goal to prove convergence results for positive operator semigroups in case that they admit certain lower bounds. Our results, in particular Theorem~\ref{thm:ind-lower-bounds-convergence} and Corollary~\ref{cor:to-ind-lower-bounds-convergence}, substantially generalise results for AL-Banach lattices that were recently obtain by M.~Gerlach and the first author \cite{GerlachPFConv} as well as a result of Ayupov, Sarymsakov and Grabarnik which was proved in \cite{Ajupov1987} for universal lower bounds on the predual of a von Neumann algebra (see also \cite[Theorem~3.3.6, Theorem~3.3.17]{Emelyanov2007}).

\medskip

In our first result of this type, Theorem~\ref{thm:ind-lower-bounds-which-are-fixed-points-imply-convergence}, we consider the special case where the lower bounds are, at the same time, fixed points of the semigroup.

\begin{theorem} \label{thm:ind-lower-bounds-which-are-fixed-points-imply-convergence}
	Let $X$ be an AN-space and let $\T = (T_t)_{t \in J}$ be a bounded positive one-parameter semigroup on $X$ (where $J = \N_0$ or $J = [0,\infty)$). The following assertions are equivalent:
	\begin{enumerate}
		\item[(i)] The semigroup $\T$ is strongly convergent and the limit operator $P$ fulfils \\$\|Pf\| \ge \gamma \|f\|$ for a number $\gamma > 0$ and all $f \in X_+$.
		\item[(ii)] For every $f \in X_+$ of norm $1$ we can find a lower bound $h_f \in X_+$ of $(\T,f)$ which is a fixed point of $\T$ and for which we have $\inf\{\|h_f\|: \; f \in X_+, \; \|f\| = 1\} > 0$.
	\end{enumerate}
\end{theorem}

For the proof of the non-trivial implication in Theorem~\ref{thm:ind-lower-bounds-which-are-fixed-points-imply-convergence} we borrow an essential idea from the proof of \cite[Theorem~4.4]{GerlachLB}. However, the proof of the latter result cannot be completely transferred to our situation since, at some steps, it heavily uses the lattice structure of the underlying space. As our space does not need to be a lattice, we thus need some new arguments.

The following observation which is, in a sense, a particularly neat version of Proposition~\ref{prop:asymptotic-domination-via-error-function}, will come in handy in the proof of Theorem~\ref{thm:ind-lower-bounds-which-are-fixed-points-imply-convergence}.

\begin{rem} \label{rem:good-approximation-on-base-norm-spaces}
	Assume that $X$ is a base norm space and let $\varepsilon > 0$. If a vector $f \in X$ fulfils $\distPos(f) < \varepsilon$, then we can find vectors $g,e \in X_+$ such that $f = g - e$ and $\|e\| < \varepsilon$.
\end{rem}
\begin{proof}
	There exists a vector $u \in X_+$ such that $\|f-u\| < \varepsilon$ and by formula~\eqref{eq:base-norm-formula} we can find $d,e \in X_+$ for which we have $\|d\|+\|e\| < \varepsilon$ and $f-u = d-e$. Hence, $f = u+d-e$, so the assertion follows with $g := u+d$.
\end{proof}

\begin{proof}[Proof of Theorem~\ref{thm:ind-lower-bounds-which-are-fixed-points-imply-convergence}]
	The implication ``(i) $\Rightarrow$ (ii)'' is obvious. For the converse implication ``(ii) $\Rightarrow$ (i)'' it suffices to prove that (ii) implies strong convergence of the semigroup. Set $\beta := \inf\{\|h_f\|: \; f \in X_+, \; \|f\| = 1\} > 0$. We proceed in three steps.
	
	(1) According to Corollary~\ref{cor:ind-lower-bouds-renorming} we may assume that $X$ is even a base norm space, meaning that the norm fulfils the equality~\eqref{eq:base-norm-formula}; we may also assume throughout the proof that $\T$ is a Markov semigroup.

	(2) Let $\varepsilon > 0$ and let $f \in X_+$ be of norm $\|f\| = 1$. We show that there exists a vector $h_{f,\varepsilon} \in X_+$ such that $\|h_{f,\varepsilon}\| \ge 1$ and such that $\limsup_{t \to \infty} \distPos(T_tf - h_{f,\varepsilon}) \le \varepsilon$. To this end, first note that $\|h_f\| \le 1$; this follows from Proposition~\ref{prop:asymptotic-domination-via-error-function} and from the fact that $\|T_tf\| = 1$ for all $t \in J$. Hence the number $\delta := 1 - \|h_f\|$ is non-negative. Let us construct an increasing sequence $(h_n)_{n \in \N_0} \subseteq E_+$ of fixed points of $\T$ which fulfils the following two properties:
	\begin{enumerate}[label=(\alph*)]
		\item For each $n \in \N_0$ we have $\|h_n\| \ge 1 - \delta (1-\beta)^n$.
		\item For each $n \in \N_0$ we have $\limsup_{t \to \infty} \distPos(T_tf - h_n) < \varepsilon$.
	\end{enumerate}
	We choose $h_0 := h_f$, which clearly fulfils~(a) and~(b) and which is a fixed point of $\T$ be assumption. Assume now that, for some $n \in \N_0$, the vector $h_n$ has already been constructed. If $\|h_n\| \ge 1$, then we simply define $h_{n+1} := h_n$, which clearly fulfils~(a) and~(b). If $\|h_n\| < 1$, then we proceed as follows. According to property~(b), there exists a time $t_0 \in J$ such that $\distPos(T_{t_0}f - h_n) < \varepsilon$. Remark~\ref{rem:good-approximation-on-base-norm-spaces} thus implies the existence of vectors $g_n,e_n \in X_+$ which fulfil $\|e_n\| < \varepsilon$ and $T_{t_0}f - h_n = g_n - e_n$. Note that we have
	\begin{align*}
		\|g_n\| + \|h_n\| = \|g_n + h_n\| = \|T_{t_0}f + e_n\| \ge \|T_{t_0}f\| = 1,
	\end{align*}
	so $\|g_n\| \ge 1-\|h_n\| > 0$. We define $a_n := \|g_n\| \cdot h_{g_n/\|g_n\|}$ and $h_{n+1} := h_n + a_n$. As $a_n$ and $h_n$ are fixed points of $\T$, so is $h_{n+1}$. We have to show that $h_{n+1}$ fulfils~(a) and~(b). Observe that
	\begin{align*}
		\|h_{n+1}\| & = \|h_n\| + \|a_n\| \ge \|h_n\| + \|g_n\| \beta \ge \|h_n\| + (1-\|h_n\|) \beta \\
		& = \beta + \|h_n\|(1-\beta) \ge \beta + \big( 1 - \delta (1-\beta)^n \big) (1-\beta) = 1 - \delta (1-\beta)^{n+1};
	\end{align*}
	for the last inequality we used that $\beta \le 1$, which follows from the fact that the vector $h_f$ has norm at most $1$. Hence, $h_{n+1}$ fulfils~(a). To see~(b), we let $s\in J$ and compute that
	\begin{align*}
		T_{t_0+s}f - h_{n+1} = T_s (h_n+g_n-e_n) - T_s h_n - a_n = T_sg_n - a_n - T_se_n
	\end{align*}
	(where we used for the first equality that $h_n$ is a fixed point of $\T$) and hence,
	\begin{align*}
		\distPos(T_{t_0+s}f - h_{n+1}) & \le \distPos(T_s g_n - a_n) + \|T_se_n\| \to \|e_n\|
	\end{align*}
	as $s \to \infty$. Therefore, we have $\limsup_{t \to \infty} \distPos(T_tf - h_{n+1}) \le \|e_n\| < \varepsilon$, i.e.\ $h_{n+1}$ fulfils~(b).
	
	Now, as we have the sequence $(h_n)_{n \in \N_0}$ available, it is easy to obtain a vector $h_{f,\varepsilon}$ with the properties that we claimed above. Indeed, note that we have $\|h_n\| < 1 + \varepsilon$ for each $n \in \N_0$; this follows from property~(b) and from Remark~\ref{rem:good-approximation-on-base-norm-spaces} above. Since the sequence $(h_n)_{n \in \N_0} \subseteq X_+$ is increasing and norm-bounded it is, as the norm is additive on $X_+$, a Cauchy sequence; thus it converges to a vector $h_{f,\varepsilon} \in X_+$. We conclude from~(a) that $h_{f,\varepsilon}$ has norm at least $1$ and from~(b) that $\limsup_{t \to \infty} \distPos(T_tf - h_{f,\varepsilon}) \le \varepsilon$ as claimed.
	
	(3) Now, consider a vector $f \in X_+$ of norm $\|f\| = 1$. We show that the net $(T_t f)_{t \in J}$ is a Cauchy net in $X$ and hence convergent; this implies the assertion. So, let $\varepsilon >0$. By what we have shown above, there exists a time $t_0 \in J$ such that $\distPos(T_tf - h_{f,\varepsilon}) < 2\varepsilon$ for all $t \ge t_0$. Fix $t \ge t_0$. According to Remark~\ref{rem:good-approximation-on-base-norm-spaces} we can decompose $T_tf - h_{f,\varepsilon}$ as
	\begin{align*}
		T_tf - h_{f,\varepsilon} = g - e,
	\end{align*}
	where $g,e \in X_+$ and where $\|e\| < 2\varepsilon$. Hence, we have $T_tf + e = h_{f,\varepsilon} + g$ and therefore,
	\begin{align*}
		1 + 2\varepsilon > 1 + \|e\| = \|T_tf + e\| = \|h_{f,\varepsilon}\| + \|g\| \ge 1 + \|g\|,
	\end{align*}
	which proves that $\|g\| < 2\varepsilon$, too. We conclude that $\|T_tf - h_{f,\varepsilon}\| = \|g-e\| < 4\varepsilon$. Since $t \ge t_0$ was arbitrary we obtain $\|T_tf - T_sf\| < 8\varepsilon$ for all $t,s \ge t_0$. This proves that $(T_tf)_{t \in J}$ is indeed a Cauchy net, as claimed.
\end{proof}

The assumption in Theorem~\ref{thm:ind-lower-bounds-which-are-fixed-points-imply-convergence} that the lower bounds $h_f$ be fixed points is, of course, a bit peculiar. The easiest way to overcome it is to require that the semigroup is mean ergodic:

\begin{corollary} \label{cor:mean-ergodic-stable}
	Let $X$ be an AN-space. Let $\T = (T_t)_{t \in J}$ be a bounded positive one-parameter semigroup on $X$ (where $J = \N_0$ or $J = [0,\infty)$) and assume that $\T$ is strongly continuous in case $J = [0,\infty)$. Suppose that the following hypotheses hold:
	\begin{enumerate}
		\item For every $f \in X_+$ of norm $1$ there exists a lower bound $h_f \in X_+$ of $(\T,f)$ such that $\beta := \inf_{f} \|h_f\| > 0$,
		\item $\T$ is mean ergodic.
	\end{enumerate}
	Then $\T$ is strongly convergent.
\end{corollary}
\begin{proof}
	By Proposition~\ref{prop:renorming-to-become-a-base-norm-space} and Corollary \ref{cor:ind-lower-bouds-renorming} we can assume that $X$ is a base norm space and that $\T$ is a Markov semigroup. By hypothesis~(ii) the net \\$(A_t(\T)h_f)_{t > 0}$ converges to a $\T$--in\-va\-ri\-ant vector $g_f \in X_+$. Since all $A_t(\T)$ are Markov operators, we have $\|g_f\| \ge \beta$. Moreover, $g_f$ is a lower bound for $(\T,f)$ as the set of all lower bounds for $(\T,f)$ is $\T$-invariant (as a consequence of formula~\eqref{eq:positive-operator-and-distance-to-cone}), convex and closed. Now the assertion follows from Theorem~\ref{thm:ind-lower-bounds-which-are-fixed-points-imply-convergence}
\end{proof}

In the subsequent Theorem~\ref{thm:ind-lower-bounds-convergence} we give a deeper sufficient condition in order to get rid of the assumption that every lower bound $h_f$ be a fixed point of $\T$. For its proof we need also the following lemma:

\begin{lemma}\label{lem:use-of-bidual}
	Let $X$ be an AN--space which is a projection band in its bi-dual $X''$. Let $\T = (T_t)_{t \in \N_0} = (T^n)_{n \in \N_0}$ be a semigroup of Markov operators and let $f \in X_+$ be a vector of norm $1$.
	
	If $h_f \in X_+$ is a (mean) lower bound of $(\T,f)$, then there also exists a (mean) lower bound $g_f \in X_+$ of $(\T,f)$ which is a fixed point of $\T$ and which fulfils $\|g_f\| = \|h_f\|$.
\end{lemma}
\begin{proof}
	We consider $X$ as a subspace of its bidual $X''$ by means of evaluation. The sequence of Ces{\`a}ro means $(A_t(\T) f)_{t \in \N}$ has a subnet which converges with respect to the weak${}^*$-topology on $X''$ (i.e.\ the topology on $X''$ which is induced by $X'$) to a fixed point $x'' \in X''_+$ of the bi-dual semigroup $\T'' := (T_t'')_{t \in \N_0}$. One readily checks that $x'' \ge h_f$.
	
	Now, let $P \in \L(X'')$ be the band projection onto $X$. Then we have $0 < h_f = Ph_f \le Px''=:y$, i.e.\ $h_f$ is contained in the order interval $[0,y]$ in $X$. On the other hand, we have $T_t y = P T_t'' Px'' \le P T_t'' x'' = Px'' = y$ for all $t \in \N_0$, so the order interval $[0,y]$ in $X$ is invariant under the action of $\T$. Thus, the orbit of $h_f$ under $\T$ is contained in $[0,y]$. Since $[0,y]$ is weakly compact according to Proposition~\ref{prop:band-in-bidual-implies-weakly-compact-intervals}, we conclude from the mean ergodic theorem (see e.g.\ \cite[Theorem 1.1.7]{Emelyanov2007}) that the sequence of Ces{\`a}ro means $(A_t(\T) h_f)_{t \in \N}$ converges in norm to a fixed point $g_f \in X_+$ of $\T$. Note that $\|g_f\| = \|h_f\|$ since the semigroup $\T$ consists of Markov operators. Moreover, the set of all (mean) lower bounds for $(\T,f)$ is easily checked to be $\T$-invariant, convex and norm closed. Thus, $g_f$ is a (mean) lower bound of $(\T,f)$.
\end{proof}

\begin{theorem} \label{thm:ind-lower-bounds-convergence}
	Let $X$ be an AN--space and let $\T = (T_t)_{t \in J}$ be a bounded positive one-parameter semigroup on $X$ (where $J = \N_0$ or $J = [0,\infty)$). Assume that the following hypotheses hold:
	\begin{enumerate}
		\item[(i)] For every $f \in X_+$ of norm $1$ there exists a lower bound $h_f \in X_+$ of $(\T,f)$ such that $\inf\{\|h_f\|: \; f \in X_+, \; \|f\| = 1\} > 0$.
		\item[(ii)] The space $X$ is a projection band in its bi-dual.
	\end{enumerate}
	Then $\T$ is strongly convergent.
\end{theorem}

We point out that the condition $\inf\{\|h_f\|: \; f \in X_+, \; \|f\| = 1\} > 0$ in hypothesis~(i) of the above theorem cannot be dropped, in general -- even if $X$ is an $L^1$-space; a counterexample can be found in \cite[Example~4.3]{GerlachLB}. Before we prove Theorem~\ref{thm:ind-lower-bounds-convergence} we present a couple of important special cases:

\begin{corollary}\label{cor:to-ind-lower-bounds-convergence}
	Let $X$ be an AN--space and let $\T = (T_t)_{t \in J}$ be a bounded positive one-parameter semigroup on $X$ (where $J = \N_0$ or $J = [0,\infty)$). Assume that hypothesis~(i) from Theorem~\ref{thm:ind-lower-bounds-convergence} holds.
	
	Then each of the following conditions ensures that $\T $ is strongly convergent:
	\begin{enumerate}	
		\item[(a)] $X$ is reflexive.
		\item[(b)] $X$ is an AL-Banach lattice (for instance $X = L^1(\Omega, \Sigma, \mu)$ where $(\Omega, \Sigma ,\mu)$ is a measure space).
		\item[(c)] $X$ is the predual of a von Neumann algebra, or a noncommutative $L^1$-sspace with respect to a $\sigma$-finite trace.
	\end{enumerate}
\end{corollary}
\begin{proof}
	Assertion (a) trivially follows from Theorem~\ref{thm:ind-lower-bounds-convergence}. Assertions~(b) follows from the fact that each AL-Banach lattice is a so-called KB-space and is thus a projection band in thus bi-dual (see for instance Theorem~2.4.12, Corollary~2.4.13 and Theorem~1.2.9 in \cite{Meyer-Nieberg1991}); alternatively, (b) follows from (c) as every AL-Banach lattice is the predual of commutative von Neumann algebra (see \cite[Theorem~II.9.3]{Schaefer74}). Assertion~(c) follows from \cite[Proposition~1.17.7]{Sakai1971}.
\end{proof}	

The assertion about AL-Banach lattices in part~(b) of the above corollary was proved by Gerlach and the first of the present authors in \cite[Theorem~4.4 and Corollary~4.5]{GerlachLB}. Part~(c) of the above corollary shows that the same result remains true on preduals of von Neumann algebras.

\begin{proof}[Proof of Theorem~\ref{thm:ind-lower-bounds-convergence}]
	By Proposition~\ref{prop:renorming-to-become-a-base-norm-space} and Corollary~\ref{cor:ind-lower-bouds-renorming} we may assume throughout the
proof that $X$ is a base norm space and that $\T$ is a Markov semigroup.

	First let $J = \N_0$. By Lemma~\ref{lem:use-of-bidual} we may assume that each lower bound $h_f$ in hypothesis~(i) is a fixed point of the semigroup $\T$. Thus, the assertion follows from Theorem~\ref{thm:ind-lower-bounds-which-are-fixed-points-imply-convergence}.
	
	Now, let $J = [0, \infty)$. For each $t>0$ the discrete semigroup $(T_{nt})_{n \in\N_0}$ also fulfils the hypotheses of the theorem; hence, by what we have just proved, this semigroup is strongly convergent. According to \cite[Theorem~2.1]{GerlachLB} this implies that the semigroup $(T_t)_{t \in [0,\infty)}$ is strongly convergent, too.
\end{proof}

The following question still remains open.

\begin{open_problem}
	Does Theorem~\ref{thm:ind-lower-bounds-convergence} remain true if we do not assume hypothesis (ii)? If not, does it remain true in case that we only assume that all order intervals in $X$ be weakly compact?
\end{open_problem}

Let us close this subsection with another consequence of Theorem~\ref{thm:ind-lower-bounds-which-are-fixed-points-imply-convergence}. The following corollary is a generalisation of \cite[Corollary~4.6]{GerlachLB} where the same result was proved only on AL-Banach lattices. The result is remarkable in the sense that it concludes convergence of a \emph{dominating} semigroup from convergence of the dominated semigroup -- while most related results in the literature (see for instance \cite{Emelyanov2001} as well as our results in Section~\ref{section:long-term-behaviour-of-asymp-dom-sg}) work the other way round.

\begin{corollary} \label{cor:dominating-sg-converges}
	Let $X$ be an AN--space and let $J = \N_0$ or $J = [0,\infty)$. Let  $\T = (T_t)_{t \in J}$ and $\SS = (S_t)_{t \in J}$ be two bounded positive semigroups on $X$ such that $\SS$ asymptotically dominates $\T$.
	
	If $\T$ is strongly convergent and if the limit operator $P_\T := \lim_{t \to \infty} T_t$ fulfils $\|P_\T f\| \ge \beta \|f\|$ for a number $\beta > 0$ and all $f \in X_+$, then $\SS$ is also strongly convergent.
\end{corollary}

Note that the assumption $\|P_\T f\| \ge \gamma \|f\|$ for all $f \in X_+$ in the above corollary is automatically fulfilled for $\gamma = 1$ if $\T$ is a Markov semigroup. To derive Corollary~\ref{cor:dominating-sg-converges} from Theorem~\ref{thm:ind-lower-bounds-which-are-fixed-points-imply-convergence} we need the following simple result.

\begin{lemma} \label{lem:base-norm-spaces-domination-implies-convergence}
	Le $X$ be an AN--space, let $J = \N_0$ or $J = [0,\infty)$ and consider a mapping $f: J \mapsto X_+$. If $f$ asymptotically dominates a constant vector $x \in X_+$ of norm $\|x\| = 1$ and if $\limsup_{t \to \infty} \|f(t)\| \le 1$, then $f(t)$ converges to $x$ as $t \to \infty$.
\end{lemma}
\begin{proof}
	Let $\eins \in X'$ denote the norm functional on $X$. Since $f \das x$, Proposition~\ref{prop:asymptotic-domination-via-error-function} shows that there exists a mapping $r: J \to X_+$ satisfying $\lim_{t \to \infty} r(t) = 0$ as well as $x \le f(t) + r(t)$ for all $t \in J$. This inequality implies $\liminf_{t \to \infty}\|f(t)\| \ge \|x\| = 1$, hence $\lim_{t \to \infty} \|f(t)\| = \|x\|$. Using again that $f(t) + r(t) - x$ is positive, we obtain
	\begin{align*}
		\|f(t)-x\| & \le \|f(t) + r(t) - x\| + \|r(t)\| \\
		& = \langle \eins, f(t) \rangle + \langle \eins, r(t) \rangle - \langle \eins, x\rangle + \|r(t)\| \\
		& = \|f(t)\| - \|x\| + 2\|r(t)\| \to 0
	\end{align*}
	as $t \to \infty$. This proves the assertion.
\end{proof}

\begin{proof}[Proof of Corollary~\ref{cor:dominating-sg-converges}]
	For each $f \in X_+$ of norm $1$ the vector $P_\T f$ is a lower bound for $(\T,f)$ and thus also for $(\SS,f)$. Moreover, we have $\inf\{\|Pf\|: \; f \in X_+, \; \|f\| = 1\} \ge \beta > 0$. According to Corollary~\ref{cor:ind-lower-bouds-renorming} we may thus endow $X$ with an equivalent norm $\|\cdot\|_1$ such that $(X,\|\cdot\|_1)$ becomes a base norm space and such that $\SS$ becomes a Markov semigroup with respect to $\|\cdot\|_1$.
	
	Of course we have $\|P_\T f\|_1 \ge \tilde \gamma \|f\|_1$ for a constant $\tilde \gamma > 0$ and for all $f \in X_+$. Moreover, $P_\T f$ is also a lower bound for $(\T,f)$ and $(\SS,f)$ with respect to the norm $\|\cdot\|_1$ for each $f \in X_+$ which satisfies $\|f\|_1 = 1$.
	
	Now, fix $f \in X_+$ such that $\|f\|_1 = 1$. Clearly, $P_\T f$ is a fixed point of $\T$. Now we show that it is also of fixed point of $\SS$: the net $(S_tP_\T f)_{t \in J}$ asymptotically dominates the constant net $(T_tP_\T f)_{t \in J} = (P_\T f)_{t \in J}$. Since $\SS$ is a Markov semigroup with respect to $\|\cdot\|_1$ we have $\|S_t P_\T f\|_1 = \|P_\T f\|_1$ for all $t \in J$, so it follows from Lemma~\ref{lem:base-norm-spaces-domination-implies-convergence} that $S_t P_\T f$ converges to $P_\T f$ as $t \to \infty$. Hence, $P_\T f$ is a fixed point of $\SS$.
	
	Hence, the assertion follows from the implication ``(ii) $\Rightarrow$ (i)'' in Theorem~\ref{thm:ind-lower-bounds-which-are-fixed-points-imply-convergence}.
\end{proof}

\subsection{Universal lower bounds and strong convergence}

Let us now consider the case where the lower bound $h$ of $(\T,f)$ is universal. We start with a technical lemma:

\begin{lemma} \label{lem:fixed-space}
	Let $J = \N_0$ or $J = [0,\infty)$ and let $\T = (T_t)_{t \in J}$ be a Markov semigroup on an AN--space $X$. Assume that $\T$ is strongly continuous in case that $J = [0,\infty)$ and suppose that there exists a universal mean lower bound $h > 0$ for $\T$. Then the following assertions hold.
	\begin{enumerate}
		\item[(a)] The fixed space $\fix(\T')$ of the dual semigroup is spanned by the norm functional $\one \in X'$.
		\item[(b)] The space $\overline{N(\T)} := \{x \in X: \; \lim_{t \to \infty}A_t(\T)x = 0\}$ is equal to $\ker(\one)$.
	\end{enumerate}
\end{lemma}
\begin{proof}
	(a) Assume for a contradiction that $\fix(\T')$ is at least two-dimensional. Then we find a vector $\varphi_1 \in \fix(\T')$ which is linearly independent of $\one$ and has norm at most $1/2$; by Proposition~\ref{prop:properties-of-norm-functional}(b) we have $\varphi_1 \in [-\frac{\one}{2}, \frac{\one}{2}]$. We define $\varphi_2 := \frac{\one}{2} + \frac{\varphi_1}{2}$ and thus obtain a vector $\varphi_2 \in \fix(\T')$ which is linearly independent of $\one$ and contained in $[0,\one]$. Set $\hat \alpha = \sup\{\alpha \in \R: \; \varphi_2 - \alpha \one \ge 0\} \in [0,1]$ and define $\varphi_3 := \varphi_2 - \hat \alpha \one$. Again, $\varphi_3$ is a vector in $\fix(\T')$ which is linearly independent of $\one$ and which fulfils $\varphi_3 \in [0,\one]$; moreover, we have $\varphi_3 \not= 0$ and $\varphi_3 - \varepsilon \one \not\ge 0$ for any $\varepsilon > 0$. Set $\hat \beta := \sup\{\beta \in [0,\infty): \; \beta \varphi_3 \in [0,\one]\} \in [1,\infty)$ and finally define $\varphi_4 := \hat \beta \varphi_3$. Then $\varphi_4$ is a vector in $\fix(\T')$ which is linearly independent of $\one$, which is contained in $[0,\one]$ and which fulfils $\varphi_4 \pm \varepsilon \one \not\in [0,\one]$ for any $\varepsilon > 0$.
	
	If we set $\varphi_5 = \one - \varphi_4$, then $\varphi_5$ has exactly the same properties that we have just listed for $\varphi_4$. We have $\langle \varphi_4,h\rangle + \langle\varphi_5, h \rangle = \|h\| > 0$, hence one of the summands does not vanish. Without loss of generality we may thus assume that $\delta := \langle \varphi_4, h \rangle > 0$.

	For all $t \in J$ we have $T_t'\varphi_4 = \varphi_4$; for $x \in X_+$ and $t > 0$ this yields $\langle \varphi_4, x\rangle = \langle \varphi_4, A_s(\T) x\rangle$ and hence
	\begin{align*}
		\langle \varphi_4,x\rangle = \lim_{t \to \infty}\langle \varphi_4, A_s(T)x\rangle \ge \langle \varphi_4, h\rangle \|x\| = \langle \delta \one, x \rangle.	
	\end{align*}
	Thus $\varphi_4 \ge \delta \one$, which is a contradiction.

	(b) By Yosida's ergodic theorem (see \cite[Theorem~1.1.9]{Emelyanov2007} for the time discrete case; the time continuous case is similar) the space $\overline{N(\T)}$ coincides with the annihilator of $\fix(\T')$, which in turn equals the kernel of $\one$ by~(a).
\end{proof}

With the aid of this lemma  we obtain another interesting consequence of Theorem~\ref{thm:ind-lower-bounds-convergence}

\begin{theorem} \label{thm:unif-lower-bounds-convergence}
	Let $X$ be an ordered Banach space with generating cone and let $J = \N_0$ or $J = [0,\infty)$. Let $\T = (T_t)_{t \in J}$ be a positive and bounded semigroup on $X$.
	
	If there exists a universal lower bound $h > 0$ of $\T$, then each of the following conditions implies that $T_t$ converges strongly to a rank-$1$ projection as $t \to \infty$:
	\begin{enumerate}[label=(\alph*)]
		\item\label{thm:unif-lower-bounds-convergence:item:band-in-bidual} The space $X$ is a projection band in its bi-dual.
		
		\item\label{thm:unif-lower-bounds-convergence:item:l_1-space} The space $X$ is an AL-Banach lattice or the pre-dual of a von-Neumann algebra.
		
		\item\label{thm:unif-lower-bounds-convergence:item:reflexive} The space $X$ is reflexive.
		
		\item\label{thm:unif-lower-bounds-convergence:item:mean-ergodic} We have $J = \N_0$ and $\T$ is mean ergodic, or we have $J = [0,\infty)$ and $\T$ is a mean ergodic $C_0$-semigroup.
		
		\item\label{thm:unif-lower-bounds-convergence:item:non-trivial-fixed-space} The fixed space of $\T$ contains a non-zero element.
	\end{enumerate}
\end{theorem}
\begin{proof}
	According to Corollary \ref{cor:unif-mean-lower-bounds-renorming} we may assume that $X$ is an AN--space (even a base norm space). We first prove that $\T$ is strongly convergent.
	
	If~(a) is fulfilled, strong convergence of $\T$ follows from Theorem~\ref{thm:ind-lower-bounds-convergence}; if~(b) or~(c) is fulfilled, strong convergence follows from Corollary~\ref{cor:to-ind-lower-bounds-convergence} and if~(d) is fulfilled it follows from Corollary~\ref{cor:mean-ergodic-stable}.
	
	Now, let~(e) be fulfilled. First assume that $J = \N_0$ and let $x \in \fix(\T) \setminus \{0\}$. Then $x = A_n(\T)x$ does not converge to $0$ as $n \to \infty$, so it follows from Lemma~\ref{lem:fixed-space}(b) that $x$ is not contained in the kernel of $\one$. Hence, $\langle \one,x\rangle \not= 0$, so $\fix(\T)$ separates the span of $\one$, which coincides with the fixed space of $\T'$ according to Lemma~\ref{lem:fixed-space}(a). Therefore, $\T$ is mean ergodic according to Sine's mean ergodic theorem (see e.g.\ \cite[Theorem~1.1.11]{Emelyanov2007}); thus, the assertion follows from~(e).
	
	Now, let $J = [0,\infty)$. For each $t > 0$ the time-discrete semigroup $(T_{nt})_{n \in \N_0}$ has the universal lower bound $h > 0$ and fulfils assumption~(e); thus, $(T_{nt})_{n \in \N_0}$ is strongly convergent. According to \cite[Theorem~2.1(a)]{GerlachLB} this already implies that $(T_t)_{t \in [0,\infty)}$ is strongly convergent.
	
	Finally we have to show that the strong limit $P_\T := \lim_{t \to \infty} T_t$ has rank $1$. Note that $P_\T = \lim_{n \to \infty} T_{n}$, so the range of the dual projection $P_\T'$ coincides with the fixed space of the semigroup $\T_1' := (T'_n)_{n \in \N_0}$. Since $h > 0$ is a universal lower bound for the semigroup $(T_n)_{n \in \N_0}$, it is also a universal mean lower bound for the same semigroup; it thus follows from Lemma~\ref{lem:fixed-space}(a) that the fixed space of $\T_1'$ is spanned by $\one$. Hence, the range of $P_\T'$ is one-dimensional and thus, so is the range of $P_\T$.
\end{proof}

The following question remains open.

\begin{open_problem}
	Does Theorem~\ref{thm:unif-lower-bounds-convergence} remain true if we do not assume any of the conditions~(a)--(e)? If not, does the theorem remain true without these assumptions at least in case that every order interval in $X$ is weakly compact?
\end{open_problem}

Another consequence of Lemma~\ref{lem:fixed-space} is the following \emph{mean} lower bound condition for mean ergodicity.

\begin{theorem} \label{thm:unif-mean-lower-bound-mean-ergodic}
	Let $X$ be an ordered Banach space with generating cone and let $J = \N_0$ or $J = [0,\infty)$. Let $\T = (T_t)_{t \in J}$ be a positive and bounded semigroup on $X$ and assume that $\T$ is a $C_0$-semigroup in case that $J = [0,\infty)$.
	
	If there exists a universal mean lower bound $h > 0$ of $\T$ and if $\T$ has a non-zero fixed point, then $\T$ is mean ergodic and the corresponding mean ergodic projection has rank $1$.
\end{theorem}
\begin{proof}
	Let $x \in \fix(\T) \setminus \{0\}$. Since $x = A_t(\T)x$ does not converge to $0$ as $t \to \infty$, it follows from Lemma~\ref{lem:fixed-space}(b) that $x$ is not contained in $\ker(\one)$. Therefore, $\langle \one,x\rangle \not= 0$ and thus $\fix(\T)$ separates the span of $\one$, which is in turn equal to $\fix(\T')$ by Lemma~\ref{lem:fixed-space}(a). Hence, $\T$ is mean ergodic
\end{proof}

As is already well--known the assumption in Theorem~\ref{thm:unif-mean-lower-bound-mean-ergodic} that $\T$ possesses a non-zero fixed point is redundant in each of the following two cases:
\begin{itemize}
	\item $X$ is an AL-Banach lattice (see see e.g.~\cite[Theorem~5]{Emelyanov2004} or~\cite[Theorem~3.2.2]{Emelyanov2007});
	\item $X$ is the pre-dual of a von Neumann algebra and moreover all operators are \emph{completely} positive (see \cite[Theorem~3]{Emelyanov2006} or \cite[Theorem~3.3.7]{Emelyanov2007});
\end{itemize}
(note that the second case includes the first one since the dual of an AL-Banach lattice is a commutative von Neumann algebra, cf.~\cite[Theorem~II.9.3]{Schaefer74}). The following corollary contains a generalisation of this result.

\begin{corollary} \label{cor:unif-mean-lower-bound-mean-ergodic}
	Let $X$ be an ordered Banach space with generating cone and let $J = \N_0$ or $J = [0,\infty)$. Let $\T = (T_t)_{t \in J}$ be a positive and bounded semigroup on $X$ and assume that $\T$ is a $C_0$-semigroup in case that $J = [0,\infty)$.
	
	Suppose that there exists a universal mean lower bound $h > 0$ of $\T$ and that at least one of the following assertions is fulfilled:
	\begin{enumerate}
		\item[(a)] $X$ is a projection band in its bi-dual; in particular
		 $X$ is an AL-Banach lattice or a pre-dual of a von-Neumann algebra.
		\item[(b)] $X$ is the dual of an ordered Banach space $Y$ and each operator $T_t$ has a pre-dual operator in $Y$.
	\end{enumerate}
	Then $\T$ is mean ergodic and the corresponding mean ergodic projection has rank $1$.
\end{corollary}
\begin{proof}
	According to Corollary~\ref{cor:unif-mean-lower-bounds-renorming} we may assume that $X$ is a base norm space and that $\T$ is a Markov semigroup.
	
	(a) Let $f \in X_+$ be a vector of norm $1$. We apply the Ces{\`a}ro means $A_t(\T)$ to $f$ and consider the net $(A_t(\T) f)_{t \in J}$ in the bi-dual space $X''$. This net has a subnet which converges to a vector $f_0'' \in X''$. Then $f_0''$ is a fixed point of the bi-dual semigroup $\T'' = (T_t'')_{t \ge 0}$ and, as $h$ is a universal mean lower bound of $\T$, we have $f_0'' \ge h$.
	
	Now, let $P \in \L(X'')$ denote the band projection onto $X$. Then we have $Pf_0'' \ge Ph = h$, so $Pf_0'' > 0$. Moreover, $T_tPf_0'' = PT_tPf_0'' \le P T_t'' f_0'' = Pf_0''$ for all $t \in J$. Since $\T$ is a Markov semigroup and since the norm is additive on $X_+$, this implies that actually $T_t Pf_0'' = Pf_0''$ for all $t \in J$, so $Pf_0''$ is a non-zero fixed vector of $\T$. Hence, the assertion follows from Theorem~\ref{thm:unif-mean-lower-bound-mean-ergodic}.

The remainder follows from the fact that AL--Banach lattices as well as pre-duals of von Neumman algebras are projection bands in their bidual.
	
	(b) Let $f \in X_+ \setminus \{0\}$. Then a subnet of $(A_t(\T)f)_{t \in J}$ converges to a vector $f_0 \in X_+$ with respect to the weak${}^*$-topology. As $\T$ has a mean lower bound we have $f_0 > 0$. Moreover, $f_0$ is a fixed point of $\T$ since each operator $T_t$ is weak${}^*$-weak${}^*$ continuous. Thus, the assertion follows again from Theorem~\ref{thm:unif-mean-lower-bound-mean-ergodic}.
\end{proof}

It is an interesting question whether the assumption in Theorem~\ref{thm:unif-mean-lower-bound-mean-ergodic} that $\T$ possess a non-zero fixed point can be completely dropped:

\begin{open_problem}
	Does Theorem~\ref{thm:unif-mean-lower-bound-mean-ergodic} remain true if we do not a priori assume that the fixed space of $\T$ is non-zero? If not, does the theorem remain true without this assumption at least in case that every order interval in $X$ is weakly compact?
\end{open_problem}

\appendix

\section{Examples of AN--spaces} \label{appendix:example-additive-norm}

Motivated by the discussion in Section~4.1 we present several non-trivial examples of AN--spaces in this appendix. First we show by a rather general construction that there exist, on  one hand, infinite-dimensional examples of such spaces which are reflexive and that there exist, on the other hand, examples of such spaces where not every order interval is weakly compact (see Propositions~\ref{prop:additive-norm-on-centred-cones} and~\ref{prop:centred-cones-non-weakly-compact-order-interval}).  Furthermore we demonstrate how the cone of an AL-Banach lattice can be adapted to yield an AN--space which is no longer lattice ordered and not reflexive, but in which all order intervals are weakly compact (Example~\ref{ex:smaller-cone-in-al-space}).

Let $X$ be a Banach space and let $u \in X$, $u' \in X'$ such that $\langle u', u\rangle = 1$. Set $P = u' \otimes u$,  $Q = I-P$. Then the set
\begin{align*} \label{eq:representation-of-centred-cones}
	X_{+,u,u'} := \{y \in X: \langle u', y\rangle \ge\|Qy\|\}.
\end{align*}
is a closed generating cone in $X$ that renders $X$ an ordered Banach space. We call $X_{+,u,u'}$ a centered cone with parameters  $u$ and $u'$. Some properties of those cones were studied in detail in \cite[Section~4.1]{GlueckDISS}; in particular, it is shown in \cite[Proposition~4.1.4]{GlueckDISS} that the cone $X_{+,u,u'}$ is normal and has non-empty interior. Related constructions had already occurred earlier in the literature on several occasions, as for instance in \cite[Section~1]{Schaefer1963} and \cite{Fiedler1973, Lyubich1995, Lyubich1998}; see also the last part of Section~2.6 in \cite{Aliprantis2007}.

Let us now show that every Banach space which is endowed with a centred cone can be given an equivalent norm which is additive on the positive cone.

\begin{proposition_within_section} \label{prop:additive-norm-on-centred-cones}
	Let $X$ be a Banach space ordered by the centred cone $X_{+,u,u'}$. Then
	\begin{align*}
		\|z\|_1 := \max\{\|Pz\|,\|Qz\|\} \qquad \text{for all } z \in X
	\end{align*}
	defines a norm on $X$ which is equivalent to the original norm $\|\cdot\|$ and which is additive on  $X_{+,u,u'}$. In particular $(X, X_{+,u,u'}, \|\cdot\|_1)$ is an AN--space.
\end{proposition_within_section}
\begin{proof}
	The fact that $\|\cdot\|_1$ is equivalent to the given norm follows from $P+Q = I$, and the fact that $\|\cdot\|_1$ is additive on the cone follows from $\|z\|_1= \|Pz\| = \langle u',z \rangle $ for all $z \in X_{+,u',u}$.
\end{proof}

This proposition shows that every Banach space $X$ can be endowed with a cone and an equivalent norm such that it becomes an AN-space (and thus, after a further renorming, a base norm space, cf.~Proposition~\ref{prop:renorming-to-become-a-base-norm-space}); a related observation is briefly discussed in \cite[Example~4 on p.\,27]{Nagel1974}. On the other hand, the following general result implies that a Banach space which is ordered by a centred cone is reflexive if and only if each order interval is weakly compact:

\begin{proposition_within_section} \label{prop:centred-cones-non-weakly-compact-order-interval}
	Let $X$ be an ordered Banach space with normal cone and assume that $X_+$ has non-empty interior. Then the following assertions are equivalent:
	\begin{enumerate}
		\item The space $X$ is reflexive.
		\item Every order interval in $X$ is weakly compact.
	\end{enumerate}
\end{proposition_within_section}
\begin{proof}
	``(i) $\Rightarrow$ (ii)'' Every ordered interval in $X$ is weakly closed and, by the normality of $X_+$, norm bounded.
	
	``(ii) $\Rightarrow$ (i)'' Let $x$ be an interior point of $X_+$. Then $-x$ is an interior point of $-X_+$, and thus $x$ is an interior point of $2x- X_+$. Thus, the order interval
	\begin{align*}
		[0,2x] = X_+ \cap (2x-X_+)
	\end{align*}
	has non-empty interior. Hence, $[0,2x]$ contains a closed ball $B$ of strictly positive radius. Since $[0,2x]$ is weakly compact, so is $B$. Hence, $X$ is reflexive.
\end{proof}

Propositions~\ref{prop:additive-norm-on-centred-cones} and~\ref{prop:centred-cones-non-weakly-compact-order-interval} have the following interesting consequence:

\begin{rem_within_section} \label{rem:an-spaces-and-banach-lattices-appendix}
	Let $X$ be a non-reflexive Banach space; we endow $X$ with a centred cone $X_{+,u',u}$ and with the norm $\|\cdot\|_1$ from Proposition~\ref{prop:additive-norm-on-centred-cones}. Then $X$ is an AN--spaces according to this proposition, but it follows from Proposition~\ref{prop:centred-cones-non-weakly-compact-order-interval} that not every order interval in $X$ is weakly compact. In particular, $X$ cannot be a projection band in its bi-dual according to Proposition~\ref{prop:band-in-bidual-implies-weakly-compact-intervals}.
	
	This is a remarkable contrast to the Banach lattice case, since every AL-Banach lattice is a projection band in its dual.
\end{rem_within_section}

In view of the results in Section~\ref{section:long-term-behaviour-of-asymp-dom-sg}, ordered Banach spaces with weakly compact order intervals are of particular interest. AL-Banach lattices and more generally  pre-duals of von Neumann-algebras have this property. Naturally, the question arises whether one can find further examples of AN--spaces in which the order intervals are weakly compact.  Trivial examples of such spaces are given by direct sums of AL-Banach lattices with reflexive spaces that are endowed with a centred cone (and an appropriate norm, see Proposition~\ref{prop:additive-norm-on-centred-cones}). Let us close this appendix with a less trivial example; our construction takes the positive cone in an AL-Banach lattice and makes it somewhat smaller.

\begin{example_within_section} \label{ex:smaller-cone-in-al-space}
	Let $X$ be an AL-Banach lattice with positive cone $X_+$ and let $\dim X \ge 3$. Take $\varphi \in X'$ such that $\varphi_+,\ \varphi_- \not= 0$ (note that $X'$ is also a Banach lattice, so $\varphi_+$ and $\varphi_-$ are well-defined). Set
	\begin{align*}
		K = \{x \in X_+: \langle \varphi, x\rangle \ge 0\} =\{ x \in X_+: \langle\varphi_+,x\rangle \ge \langle \varphi_-, x\rangle\}.
	\end{align*}
	Then $K$ is a closed, convex and proper cone. We denote the order on $X$ with respect to the original cone $X_+$ by $\le$, while we denote the order with respect to the new cone $K$ by $\le_K$.
	
	Let us now show that the new cone $K$ is also generating. Consider
	\begin{align*}
		F = \{x \in X: \; |x| \land |y| = 0 \text{ for all } y \in X \text{ satisfying } \langle \varphi_+,|y|\rangle = 0\},
	\end{align*}
	where $\land$ denotes the infimum with respect to the original order $\le$ on $X$. The set $F$ is a non-zero band in $X$ and the functional $\varphi_+$ is strictly positive on $F$ while $\varphi_-$ vanishes on $F$. Hence $F_+ \subset K$. Choose $y \in F_+$ such that $\langle \varphi_+,y \rangle = 1$; then $\langle \varphi_{-},y\rangle = 0$ holds. For each $x \in X$ we have
	\begin{align*}
		x= x_+ - x_- = \underbrace{(x_+ +\langle \varphi_-, x_+\rangle y)}_{\in K}-\underbrace{\langle \varphi_-, x_+\rangle y}_{\in K} - \underbrace{(x_-+\langle \varphi_-, x_-\rangle y)}_{\in K}+\underbrace{\langle \varphi_-, x_-\rangle y}_{\in K}.
	\end{align*}
	Hence, $K$ is indeed generating in $X$.
	
	Next we note that the order intervals on $X$ with respect to $\le_K$ are contained in the order intervals with respect to $\le$. Indeed, for all $x,y \in X$ we have
	\begin{align*}
		[x,y]_{\le_K} = (x+K) \cap (y-K) \subseteq (x+X_+) \cap (y - X_+) = [x,y]_{\le}.
	\end{align*}
	Hence, all order intervals with respect to $\le_K$ are weakly compact. Clearly, the norm is additive on $K$, as it is additive on the larger cone $X_+$.
	
	Finally, we show that $(X,K)$ is not a lattice if $\dim F \ge 2$. Let $P$ be the band projection onto $F$, and set $Q := I-P$; we have $P,Q \not= 0$ since $\varphi_+,\varphi_- \not= 0$. Choose $v \in F_+ = PX_+$ with $\langle \varphi_+,v\rangle =1$ as well as $w \in Q X_+$ with $\langle \varphi_-, w\rangle = 1$ and define $u = -v+w$.
	
	Let $U$ be the set of all upper bounds of $0$ and $u$ with respect to the order $\le_K$. A vector $y \in X$ is contained in $U$ if and only if the following three conditions are fulfilled:
	\begin{enumerate}
		\item[(i)] $y \ge 0$ and $y \ge u$,
		\item[(ii)] $\langle\varphi_+,y\rangle \ge \langle\varphi_-,y\rangle$,
		\item[(iii)] $\langle \varphi_+,y-u\rangle \ge \langle \varphi_-, y-u\rangle$.
	\end{enumerate}
	Now we note that condition (i) is equivalent to
$y \ge 0 \lor u$ (where $\lor$ denotes the supremum with respect to the order $\le$); since the vector $0 \lor u$ equals $w$, condition (i) is fulfilled if and only if $y \ge w$. Moreover, condition (iii) is equivalent to $\langle \varphi_+,y\rangle \ge \langle \varphi_-,y\rangle - 2$; hence, condition~(iii) is automatically fulfilled in case that~(ii) is fulfilled.

Summing up, we conclude that $y \in U$ if and only if $y \ge w$ and $\langle \varphi_+,y\rangle \ge \langle \varphi_-,y\rangle$, which is in turn fulfilled if and only if $y \ge w$ and $\langle \varphi_+,y\rangle \ge \langle \varphi_-,y\rangle \ge 1$ (since $\langle \varphi_-,w\rangle = 1$).

	Define $U_1 = \{y\ge 0: \langle \varphi_+,y\rangle = 1, \ Qy = w\}$. Then $U_1 \subseteq U$. Moreover, every element $y_1 \in U_1$ is minimal in $U$ with respect to the order $\le_K$. Indeed, let $y_1 \in U_1$ and let $y_0 \in U$ such that $y_0 \le_K y_1$. Then, in particular, $y_0 \le y_1$. Hence we have $Qy_0 \le Qy_1$ but, on the other hand, $Qy_1 = w = Qw \le Qy_0$, so $Qy_0 = Qy_1$. Moreover, we have $Py_0 \le Py_1$ and thus $\langle \varphi_+, Py_0\rangle \le \langle \varphi_+, Py_1\rangle$ but, on the other hand, $\langle \varphi_+, Py_1\rangle = \langle \varphi_+, y_1\rangle = 1 \le \langle \varphi_+, y_0\rangle = \langle \varphi_+, Py_0\rangle$, so actually $\langle \varphi_+,Py_0\rangle = \langle \varphi_+, Py_1\rangle$. Since $\varphi_+$ is strictly positive on $F = PX$, we conclude that $Py_0 = Py_1$. We have thus proved that $y_0 = y_1$.
	
	Therefore, $U_1$ consists of minimal upper bounds of $\{0,u\}$ with respect to $\le_K$. However, one readily checks that $U_1$ has more than one element if $\dim F \ge 2$. Our assertion is proved.
\end{example_within_section}

Note that, if $X$ is infinite dimensional in the above example, then the examples yields an AN--space with weakly compact order intervals which is neither reflexive nor a lattice.

\end{document}